%% file: main.tex
\documentclass[11pt]{amsart}
\usepackage{amssymb,amsfonts,amsmath,amsthm,cite,verbatim,mathrsfs,amscd}
\usepackage{arydshln}
\usepackage{rotating}
\usepackage{chngcntr}
\usepackage{xcolor}
\usepackage{tikz}
\usepackage{graphicx}
\usepackage[all,pdf]{xy}
\usepackage{setspace}
\usepackage[left=2.5cm, right=2.5cm, top=3.5cm, bottom=3.5cm]{geometry}
\usepackage{enumerate}
\usepackage[linktocpage=true]{hyperref}
\hypersetup{colorlinks,linkcolor=blue,urlcolor=cyan,citecolor=blue}
\usepackage[normalem]{ulem}
\newtheorem{theorem}{Theorem}[section]
\newtheorem{thm}[theorem]{Theorem}
\newtheorem{prop}[theorem]{Proposition}

\newtheorem{innercustomthm}{{\bf Theorem}}
\newenvironment{customthm}[1]
  {\renewcommand\theinnercustomthm{#1}\innercustomthm}
  {\endinnercustomthm}

\newtheorem{lem}[theorem]{Lemma}
\newtheorem{rem}[theorem]{Remark}

\newtheorem{cor}[theorem]{Corollary}

\makeatletter \@addtoreset{equation}{section}

\newcommand{\qbinom}[2]{\genfrac{[}{]}{0pt}{}{#1}{#2}}

\newcommand{\Mid}{\:|\:}  %\Mid has too much space
\DeclareMathOperator*{\CT}{CT}

\begin{document}

\title[AFLT-type $q$-Baker--Forrester ex-conjecture]
{An AFLT-type generalization of the $q$-Baker--Forrester ex-conjecture}

\author{Ricardo Zihao Huang, Wenlong Jiang, Yue Zhou}

\address{School of Mathematics and Statistics, HNPLAMA, Hunan Research Center of the Basic
Discipline for Analytical Mathematics, Central South University, Changsha 410083,
People’s Republic of China
}

\email{
  \begin{tabular}{@{}ll@{}}
    \texttt{zihaohuang@csu.edu.cn}  \texttt{jiangwenlong@csu.edu.cn} 
    \texttt{zhouyue@csu.edu.cn}
  \end{tabular}}
\date{\today}


\begin{abstract}
The Habsieger--Kadell $q$-Morris constant term identity, which is equivalent to the famous $q$-Selberg integral, has been generalized in numerous ways since the 1980s.
Among these, there are two important generalizations:
(i) the $q$-Baker--Forrester ex-conjecture, which was conjectured by Baker and Forrester in 1998 and proved by K\'{a}rolyi, Nagy, Petrov and Volkov in 2015;
(ii) the AFLT-type $q$-Morris identity (equivalently, the AFLT-type $q$-Selberg integral), which was obtained by Albion, Rains and Warnaar in 2021, as a $q$-analogue of the result of Alba, Fateev, Litvinov and Tarnopolskiy (AFLT). In this paper, by the Gessel--Xin method and the Macdonald polynomials with prescribed symmetry, we unify these two generalizations. 
\noindent

\end{abstract}

\maketitle

\input{1.introduction}

\input{2.Preliminaries}

\input{3.Twofamilies}

\input{4.Somediscussion}

\input{5.proof1}

\input{6.proof2}

\input{8.proof4}

\input{7.proof3}

\subsection*{Acknowledgements}
ChatGPT (GPT-5.6 Sol) and DeepSeek (DeepSeek-V3.2) were used only after the initial draft of the manuscript had been completed. Their use was limited to checking grammar and independently testing already established theorems on specific examples. They were not used to formulate or discover mathematical results, develop mathematical ideas, construct proofs, or contribute in any other way to the mathematical substance of the paper.

This work was supported by the Science and Technology Innovation Program and Department of Education Funded Research Projects of Hunan
Province (No. 24A0025) and the National Natural Science Foundation of China (No. 12571359).

\appendix
\input{9.appendix}

%-------------------------------------

\end{document}

%% file: 1.introduction.tex
\section{Introduction}
In 1944, Atle Selberg~\cite{Selberg} gave a remarkable multiple integral
\begin{multline}\label{eq-intro-selberg}
\int_{0}^1\cdots \int_{0}^1\prod_{i=1}^nx_{i}^{\alpha-1}(1-x_i)^{\beta-1}\prod_{1\leq i<j\leq n}
|x_i-x_j|^{2\gamma}\mathrm{d}x_{1}\cdots \mathrm{d}x_{n} \\
=\prod_{i=0}^{n-1}\frac{\Gamma(\alpha+i\gamma)\Gamma(\beta+i\gamma)
\Gamma(1+(i+1)\gamma)}{\Gamma(\alpha+\beta+(n+i-1)\gamma)\Gamma(1+\gamma)},
\end{multline}
where $\alpha, \beta, \gamma$ are complex parameters such that
\[
\mathrm{Re}(\alpha)>0, \quad \mathrm{Re}(\beta)>0,\quad
\mathrm{Re}(\gamma)>-\min\{1/n,\mathrm{Re}(\alpha)/(n-1),\mathrm{Re}(\beta)/(n-1)\}.
\]
This integral is now known as the Selberg integral, and is widely regarded as one of the most fundamental and important hypergeometric integrals. The reader is referred to \cite{FW} for more details. When $n=1$, it reduces to the Euler beta integral.

In 1980, Askey~\cite{Askey} conjectured a $q$-analogue of the Selberg integral. This conjecture was proved independently by Habsieger~\cite{Habsieger} and Kadell~\cite{Kad}. Both of them turned the $q$-Selberg integral into an equivalent constant term identity
\begin{equation}\label{eq-intro-qmorris}
\CT_x \prod_{i = 1}^{n} (x_0/x_i)_a (qx_i/x_0)_b \prod_{1 \leq i < j \leq n} (x_i / x_j)_c (qx_j / x_i)_c = \prod_{i=0}^{n-1} \frac{(q)_{ic+a+b}(q)_{(i+1)c}}{(q)_{ic+a}(q)_{ic+b}(q)_c},
\end{equation} 
where $a,b,c$ are nonnegative integers. Here for $k\in\mathbb{Z}$ and an indeterminate $z$, $(z)_k=(z;q)_k:=\prod_{j=0}^\infty\frac{1-zq^j}{1-zq^{k+j}}$ is the $q$-shifted factorial, $x:= (x_1,\ldots,x_n)$ and $\CT\limits_x f(x)$ denotes the constant term of the Laurent polynomial $f(x)$ with respect to $x$. Note that we can set $x_0 = 1$ in~\eqref{eq-intro-qmorris} without changing the constant term, because of the homogeneity. Since~\eqref{eq-intro-qmorris} was first conjectured and proved for the $q=1$ case by Morris in his Ph.D. thesis~\cite{Morris1982}, we refer to ~\eqref{eq-intro-qmorris} as the Habsieger--Kadell $q$-Morris identity or simply the $q$-Morris identity. The $q$-Morris identity, or equivalently, the $q$-Selberg integral, admits numerous generalizations. In this paper, we focus on two kinds of its generalizations.

The first generalization is the $q$-Baker--Forrester ex-conjecture. For positive integers $n,m$ such that $m\leq n$, and nonnegative integers $a,b,c$, denote
\begin{equation}\label{def-intro-Fnm}
F_{n,m}(x;a,b,c):=\prod_{i=1}^{n}(x_0/x_i)_a
(qx_i/x_0)_b \prod_{1\leq i<j\leq n}
(x_i/x_j)_{c+\chi(j\leq m)}(qx_j/x_i)_{c+\chi(j\leq m)},
\end{equation}
where $\chi(\mathrm{true})=1$ and $\chi(\mathrm{false})=0$.
In 1998, in the study of the multi-component Calogero-Sutherland model, Baker and Forrester\cite[Conjecture 2.1]{BF} conjectured that
\begin{multline}\label{eq-intro-conj-qBF1}
\CT_x F_{n,m}(x;a,b,c)\\
= \prod_{i=1}^{m-1} (1-q^{(i+1)(c+1)})\prod_{i=0}^{n-1} \cfrac{(q)_{ic+a+b+\chi(i> n-m)(i-n+m)}(q)_{(i+1)c+\chi(i>n-m)(i-n+m)}}{(q)_{ic+a+\chi(i>n-m)(i-n+m)}(q)_{ic+b+\chi(i> n-m)(i-n+m)}(q)_{c+\chi(i>n-m)}}.
\end{multline}
This conjecture was proved by K\'{a}rolyi, Nagy, Petrov and Volkov \cite{KNPV} using the Combinatorial Nullstellensatz in 2015. Note that ~\eqref{eq-intro-conj-qBF1} differs slightly from the original identity in~\cite[(2.12)]{BF}, where the roles of the first $m$ variables and the last $n-m$ variables are interchanged. These two identities are equivalent. One can obtain ~\cite[(2.12)]{BF} from~\eqref{eq-intro-conj-qBF1} by substituting
\[
x_i\mapsto x_{n-i+1}^{-1} \text{ for $i=1,\ldots,n$},\quad \text{and } \quad  x_0\mapsto q x_0^{-1},
\]
in $F_{n,m}(x;a,b,c)$, and using the fact that such a substitution does not change the constant term. The constant term identity~\eqref{eq-intro-conj-qBF1} reduces to \eqref{eq-intro-qmorris} by taking $m=n$ and $c\mapsto c-1$, or $m=1$. 

The second generalization is the AFLT-type $q$-Morris constant term identity (or equivalently, the AFLT-type $q$-Selberg integral). In 2011, based on the AGT conjecture\cite{agt}, Alba, Fateev, Litvinov and Tarnopolskiy (AFLT)\cite{AFLT} obtained a generalization of the Selberg integral over a pair of Jack polynomials. In 2021, Albion, Rains and Warnaar \cite{ARW} obtained an elliptic generalization of the AFLT integral. As a corollary, they obtained a $q$-analogue of the AFLT integral, which is a $q$-Selberg integral over a pair of Macdonald polynomials\cite[Corollary 6.1]{ARW}. By the standard transformation between $q$-Selberg type integrals and $q$-Morris type constant term identities~\cite{FW,Zhou-AFLT}, their $q$-AFLT integral is equivalent to a generalization of the $q$-Morris identity; see~\eqref{eq-intro-AFLT} below. For a positive integer $n$, nonnegative integers $a,b,c$ and partitions $\lambda$ and $\mu$, denote
\begin{multline*}\label{def-intro-An(a,b,c,lambda,mu)}
A_n(a,b,c,\lambda,\mu):=\CT\limits_x x_0^{|\lambda|-|\mu|}P_{\lambda}(x^{-1};q,q^c)
P_{\mu}\Big(\Big[\frac{q^{c-b-1}-q^a}{1-q^c}x_0+\sum_{i=1}^nx_i\Big];q,q^c\Big)\\
\times
\prod_{i=1}^{n}(x_0/x_i)_a(qx_i/x_0)_b
\prod_{1\leq i<j\leq n}(x_i/x_j)_c(qx_j/x_i)_c,
\end{multline*}
where $P_\lambda(x;q,t)$ is the Macdonald polynomial, $|\lambda|:=\sum_{i\geq 1}\lambda_i$ is the size of $\lambda$, and $f[y+z]$ is the plethystic notation (see Section~\ref{subsec-2.5}). Then for $\ell(\lambda)\leq n$, 
\begin{equation}\label{eq-intro-AFLT}
\begin{split}
A_n(a,b,c,\lambda,\mu)&=(-1)^{|\lambda|}q^{\sum_{i=1}^n \binom{\lambda_i+1}{2} -cn(\lambda)} P_\lambda\Big(\Big[ \cfrac{1-q^{nc}}{1-q^c}\Big];q,q^c \Big) P_\mu\Big(\Big[ \cfrac{q^{c-b-1}-q^{a+nc}}{1-q^c}\Big];q,q^c \Big)\\
&\times \prod_{i=1}^n \prod_{j=1}^{\ell(\mu)} (q^{b+(i-j-1)c-\lambda_i+\mu_{j+1}+1})_{\mu_j-\mu_{j+1}} \prod_{i=0}^{n-1} \cfrac{(q)_{ic+a+b}(q)_{(i+1)c}}{(q)_{ic+a+\lambda_{n-i}}(q)_{ic+b-\lambda_{i+1}+\mu_1} (q)_c},
\end{split}
\end{equation}
where $n(\lambda):=\sum_{i\geq 1}(i-1)\lambda_i$. Note that both sides of ~\eqref{eq-intro-AFLT} trivially vanish when $\ell(\lambda)> n$. We refer to~\eqref{eq-intro-AFLT} as the AFLT-type $q$-Morris constant term identity. It reduces to the $q$-Morris identity~\eqref{eq-intro-qmorris} when $\lambda=\mu=\emptyset$. Many special cases of \eqref{eq-intro-AFLT} enjoy a rich history, and we refer the reader to \cite{Zhou-AFLT} for more details.

The main objective of this paper is to unify the $q$-Baker--Forrester identity~\eqref{eq-intro-conj-qBF1} and the AFLT-type $q$-Morris identity~\eqref{eq-intro-AFLT}. Throughout this paper, let $n,m$ be positive integers such that $m\leq n$, and let $\delta^m:=(m-1,\ldots,1,0)$ be the staircase partition. Let
\[
X_m:=x_1+\cdots+x_m,\  X_{n\setminus m}:=x_{m+1}+\cdots+x_n
\]
and
\[
X_m^{-1}:=x_1^{-1}+\cdots+x_m^{-1}, \ X_{n\setminus m}^{-1}:=x_{m+1}^{-1}+\cdots+x_n^{-1}
\]
be alphabets in plethystic notation. Let $a,b,c$ be nonnegative integers, and let $\lambda,\xi,\mu$ be partitions. Denote
\begin{align}\label{definition-B_nm(a,b,c,lambda,xi,mu)}
    F_{n,m}(a,b,c,\lambda,\xi,\mu):=& \CT\limits_x x_0^{|\lambda|+|\xi|-|\mu|} F_{n,m}(x;a,b,c) P_\xi\left( \left[X_{n\setminus m}^{-1}\right];q^{c+1},q^c\right)\\
    & \times P_\lambda\left(\left[X_m^{-1}+\cfrac{1-q^c}{1-q^{c+1}}X_{n\setminus m}^{-1}\right];q,q^{c+1} \right)\nonumber \\
    & \times P_\mu\left(\left[\cfrac{q^{c-b}-q^a}{1-q^{c+1}}x_0+X_m +\cfrac{q-q^{c+1}}{1-q^{c+1}}X_{n\setminus m}\right];q,q^{c+1}\right)\nonumber.
\end{align}
For compactness, we sometimes write $P_\lambda^{(q,t)}[X]$ for $P_\lambda([X];q,t)$. The main result of this paper is the next theorem.
\begin{thm}\label{main-thm}
    For nonnegative integers $a,b,c$, and partitions $\lambda$, $\xi$, $\mu$ such that $\ell(\lambda)<m$, $\ell(\mu)< m$, $\ell(\xi)\leq n-m$, let $F_{n,m}(a,b,c,\lambda,\xi,\mu)$ be defined as in~\eqref{definition-B_nm(a,b,c,lambda,xi,mu)}. If $\lambda$ and $\xi$ satisfy at least one of the three conditions:
    \[
    \text{(i) } \xi=\emptyset;\quad \text{(ii) }\lambda=\emptyset \text{ and } \xi_1\leq m; \quad \text{(iii) }|\lambda|+|\xi|\leq \min\{n-m,m\},
    \]
    then
    \begin{align}\label{eq-intro-maineq}
    &F_{n,m}(a,b,c,\lambda,\xi,\mu)\\
    &=(-1)^{|\mu|}q^{(b+1)(|\xi|+|\lambda|)+\sum_{i=1}^{\ell(\mu)}\binom{\mu_i}{2}-(c+1)n(\mu)}\nonumber\\
    &\times P_\xi^{(q^{c+1},q^c)}\Big[\cfrac{1-q^{(n-m)c}}{1-q^c}\Big] P_\lambda^{(q,q^{c+1})}\Big[ \cfrac{1-q^{nc+m}}{1-q^{c+1}}\Big] P_\mu^{(q,q^{c+1})}\Big[ \cfrac{1-q^{b+(n-1)c+m+a}}{1-q^{c+1}}\Big]\nonumber\\
    &\times \prod_{i=1}^{m-1}\prod_{j=1}^{m-1}(q^{(i-j)(c+1)-b-\mu_i+\lambda_{j+1}})_{\lambda_j-\lambda_{j+1}}
    \prod_{i=1}^{m-1} \cfrac{ (1-q^{(i+1)(c+1)}) (q)_{ic-b-1+\eta^+_{m-i}}}{(q)_{i(c+1)-b-1-\mu_i+\lambda_1}}\prod_{i=0}^{n-m} (q^{-ic-b})_{\eta_{m+i}^+}\nonumber\\
    &\times \prod_{i=0}^{n-1}\cfrac{(q)_{ic+a+b+\chi(i\geq n-m)(i-n+m)}(q)_{(i+1)c+\chi(i\geq n-m)(i-n+m)}}{(q)_{ic+a+\eta_{n-i}^+}(q)_{ic+b+\chi(i\geq n-m)(i-n+m)+\mu_{n-i}}(q)_{c+\chi(i>n-m)}}.\nonumber
   \end{align} 
   Here $\eta:=(\lambda+\delta^m,\xi)$, and $\eta^+$ is the unique partition obtained by rearranging the entries of $\eta$ in weakly decreasing order.
\end{thm}

Note that (i) the first attempt in this direction was made by Xin and Zhou~\cite{Zhou-tran}. One of their main results gives a one-row complete symmetric function generalization of $F_{n,m}(a,b,c,\emptyset,\emptyset,\mu)$~\cite[Theorem 1.1]{Zhou-tran}. However, it appears difficult to extend their method to the case of Macdonald polynomials with more than one row; (ii) the Macdonald polynomials $P_\lambda$ and $P_\xi$ occurring in $F_{n,m}(a,b,c,\lambda,\xi,\mu)$ also appear in the theory of Calogero--Sutherland--Moser models~\cite{BDF}, and in the theory of Macdonald polynomials in superspace~\cite{double-mac}; see Theorem~\ref{thm-nonmac-factorization} below. 

The first case of Theorem~\ref{main-thm} admits a simpler form, which we state separately as a corollary.
\begin{cor}\label{cor-intro-2}
    With the same notation and assumptions as in Theorem~\ref{main-thm}, we have
    \begin{align*}
    &F_{n,m}(a,b,c,\lambda,\emptyset,\mu)\\
    &=(-1)^{|\mu|}q^{(b+1)|\lambda|+\sum_{i=1}^{\ell(\mu)} \binom{\mu_i}{2}-(c+1)n(\mu)} P_\lambda^{(q,q^{c+1})}\Big[ \cfrac{1-q^{nc+m}}{1-q^{c+1}}\Big] P_\mu^{(q,q^{c+1})}\Big[ \cfrac{1-q^{b+(n-1)c+m+a}}{1-q^{c+1}}\Big] \\
    &\times \prod_{i=1}^{m-1}\prod_{j=1}^{m-1}(q^{(i-j)(c+1)-b-\mu_i+\lambda_{j+1}})_{\lambda_j-\lambda_{j+1}}
    \prod_{i=1}^{m-1} \cfrac{ (1-q^{(i+1)(c+1)}) (q)_{i(c+1)-b-1+\lambda_{m-i}}}{(q)_{i(c+1)-b-1-\mu_i+\lambda_1} }\\
    &\times \prod_{i=0}^{n-1}\cfrac{(q)_{ic+a+b+\chi(i\geq n-m)(i-n+m)}(q)_{(i+1)c+\chi(i\geq n-m)(i-n+m)}}{(q)_{ic+a+\chi(i\geq n-m)(i-n+m)+\lambda_{n-i}}(q)_{ic+b+\chi(i\geq n-m)(i-n+m)+\mu_{n-i}}(q)_{c+\chi(i>n-m)}}.\\
    \end{align*}
\end{cor}
Corollary~\ref{cor-intro-2} reduces to the  $\ell(\mu)<n$ case of the AFLT-type $q$-Morris identity~\eqref{eq-intro-AFLT} for $m=n$, and to the $q$-Baker--Forrester identity~\eqref{eq-intro-conj-qBF1} for $\lambda=\mu=\emptyset$. 

The proof of Theorem~\ref{main-thm} relies on the well-known fact that a polynomial is uniquely determined by its roots and its value at an additional point. Under the conditions of Theorem~\ref{main-thm}, we establish the theorem in the next three steps: 
\begin{enumerate}
    \item \textbf{Polynomiality and rationality:} We show that 
    \begin{enumerate}
        \item for fixed nonnegative integers $b$ and $c$, the constant term $F_{n,m}(a,b,c,\lambda,\xi,\mu)$ is a polynomial in $q^a$ of degree at most $nb+|\mu|-|\lambda|-|\xi|$;
        \item for fixed nonnegative integers $a$ and $c$, the constant term $F_{n,m}(a,b,c,\lambda,\xi,\mu)$ is a rational function in $q^b$;
        \item for fixed nonnegative integers $a$ and $b$, the expression
        \[
        (q)_c^n/(q)_{nc} \cdot F_{n,m}(a,b,c,\lambda,\xi,\mu)
        \]
        is a rational function in $q^c$.
    \end{enumerate}
    %Note that $F_{n,m}(a,b,c,\lambda,\xi,\mu)$ cannot be simultaneously viewed as a polynomial in both $q^a$ and $q^b$. This is because viewing it as a polynomial in $q^a$ requires that $b$ be a fixed nonnegative integer, and vice versa.
    \item \textbf{Roots:}
    Let $b$ and $c$ be positive integers such that $c\geq b+\mu_1$ and $b \geq |\lambda|+|\xi|+m$. For $F_{n,m}(a,b,c,\lambda,\xi,\mu)$, as a polynomial in $q^a$, we will determine all its roots. Explicitly, we show that
    \[
    F_{n,m}(a,b,c,\lambda,\xi,\mu)=0 \quad \text{ for }-a\in B_1(\eta^+)\cup B_2(\delta^m) \cup B_3(\mu).
    \]
    Here $\eta:=(\lambda+\delta^m,\xi)$, and the sets $B_1$, $B_2$, $B_3$ are defined by
    \begin{subequations}\label{def-sets-B1-B3}
    \begin{equation}\label{def-sets-B1}
    B_1(\eta^+):=\{ic+\eta^+_{n-i}+1,\ldots,ic+b : i=0,\ldots,n-1\},
    \end{equation}
    \begin{equation}\label{def-sets-B2}
    B_2(\delta^m):=\{ic+b+1,\ldots,ic+b+\delta^m_{n-i} : i=n-m+1,\ldots,n-1\},
    \end{equation}
    \begin{equation}\label{def-sets-B3}
    B_3(\mu):=\{ic+b+\delta^m_{n-i}+1,\ldots,ic+b+\delta^m_{n-i}+\mu_{n-i} : i=n-\ell(\mu),\ldots,n-1\}.
    \end{equation}
    \end{subequations}
    Note that under the given conditions on $b$ and $c$, the set $B_1\cup B_2\cup B_3$ is a disjoint union with the cardinality $|B_1(\eta^+)|+|B_2(\delta^m)|+|B_3(\mu)|=nb+|\mu|-|\lambda|-|\xi|$. 
    \item \textbf{Additional point:} Under the same assumptions on $b$ and $c$ as in Step (2), we obtain an explicit expression for $F_{n,m}(a,b,c,\lambda,\xi,\mu)$ at $-a=(n-m)c+b+1$.
\end{enumerate}

We can uniquely determine the closed-form expression for $F_{n,m}(a,b,c,\lambda,\xi,\mu)$ for all nonnegative integers $a,b,c$ by the above steps (see Lemma~\ref{lem-preliminary-a,b,c} below). The first step is routine but not trivial. The last step follows from a direct calculation. The proof of the second step, however, is quite lengthy.  

The remainder of this paper is organized as follows. 
In the next section, we give some preliminaries.  
In Section~\ref{section-two families}, we introduce two families of constant terms, and present their polynomiality, rationality, and vanishing properties. We show in this section that these properties, together with Theorem~\ref{thm-nonmac-factorization}, establish Step (1) and Step (2) in the above. 
In Section~\ref{sec-discussion}, we make some preparations for the subsequent proofs.
In Sections~\ref{sec-proof1} and \ref{sec-proof2}, we prove the properties of the two families of constant terms mentioned above.
In Section~\ref{sec-proof4}, we prove case (1) of Theorem~\ref{thm-nonmac-factorization}. At this stage, we finish Step (1) and Step (2). In the last section, we establish Step (3) and complete the proof of Theorem~\ref{main-thm}.

%% file: 2.Preliminaries.tex
\section{Preliminaries}\label{sec-preliminaries}
\subsection{$q$-shifted factorials}
For an indeterminate $z$, define the infinite $q$-shifted factorial as
\[
(z)_{\infty} = (z; q)_{\infty} := \prod_{i=0}^{\infty} (1 - zq^{i}),
\]
where, typically in this paper, we suppress the base $q$. For an integer $k$, define
\[
(z)_{k} = (z; q)_{k} := \frac{(z; q)_{\infty}}{(zq^{k}; q)_{\infty}}.
\]
Then
\[
(z)_{k} =
\begin{cases}
(1 - z)(1 - zq) \cdots (1 - zq^{k-1}) & \text{if } k > 0, \\
1 & \text{if } k = 0, \\
\dfrac{1}{(1 - zq^{k})(1 - zq^{k+1}) \cdots (1 - zq^{-1})} = \dfrac{1}{(zq^{k})_{-k}} & \text{if } k < 0.
\end{cases}
\]
Note that for integers $r$ and $k$ such that $r\leq 0<r+k$, we have $(q^r)_k=0$. For integers $n$ and $k$, define the $q$-binomial coefficient as
\[
\qbinom{n}{k}=\cfrac{(q^{n-k+1})_k}{(q)_k}.
\]
It is well known that ~\cite[Exercise 1.2]{GM}
\begin{equation}\label{e-qbinomialn}
(z)_n=\sum_{k=0}^\infty q^{k(k-1)/2}\qbinom{n}{k} (-z)^k
\end{equation}
for all integers $n$.

By convention, when $k < j$, we set $\sum_{i=j}^{k} n_i:= 0$ and $\prod_{i=j}^{k} n_i := 1$.

\subsection{Partitions and compositions}
Throughout this paper, let $\mathbb{Z}$ be the set of integers and $\mathbb{N}$ be the set of nonnegative integers. 

We say $\alpha$ is a composition if $\alpha\in\mathbb{N}^k$ for some positive integer $k$. For a composition $\alpha$, let the size $|\alpha|:=\alpha_1+\alpha_2+\cdots$ be the sum of its entries. Let $\alpha^+=(\alpha^+_1,\alpha^+_2,\ldots)$ be the unique integer sequence obtained from $\alpha$ by rearranging the $\alpha_i$ in a weakly decreasing order. 

A partition $\lambda=(\lambda_1,\lambda_2,\dots)$ is a sequence of nonnegative integers such that
$\lambda_1\geq \lambda_2\geq \cdots$ and only finitely many of the $\lambda_i$ are positive. The length of a partition $\lambda$, denoted by $\ell(\lambda)$, is the number of non-zero $\lambda_i$.  With the possible exception of finitely many zeros, we usually drop the infinite string of zeros of a partition so that $\lambda=(\lambda_1,\lambda_2,\ldots,\lambda_n)$ denotes a partition of length at most $n$. The set of partitions of length at most $l$ is denoted by $\mathcal{P}_l$. For two partitions $\lambda=(\lambda_1,\ldots,\lambda_{n_1})$ and $\xi=(\xi_1,\ldots,\xi_{n_2})$, we write $\lambda\subseteq\xi$ if $\lambda_i\leq \xi_i$ for all $i\geq 1$, write $(\lambda,\xi)$ for the composition $(\lambda_1,\ldots,\lambda_{n_1},\xi_1,\ldots,\xi_{n_2})$ and $\lambda+\xi$ for the partition $(\lambda_1+\xi_1,\lambda_2+\xi_2,\ldots)$. Here the indices $n_1,n_2$ are allowed to be greater than the lengths of $\lambda$ and $\xi$, respectively. Note that, in our definition, $\lambda=(3,2)$ and $\lambda'=(3,2,0)$ are regarded as the same partition; however, $(3,2,\xi_1,\ldots,\xi_{n_2})$ and $(3,2,0,\xi_1,\ldots,\xi_{n_2})$ are different compositions. 

The dominance order on compositions (and hence for partitions) is defined as
\[
\alpha\leq \beta \iff |\alpha|=|\beta|\text{ and }\alpha_1+\cdots+\alpha_k\leq \beta_1+\cdots+\beta_k \text{  for all }k\geq 1.
\]
The Bruhat order on compositions is defined by
\[
\alpha \preceq \beta \iff \alpha^{+} < \beta^{+} \text { or in the case } \alpha^{+}=\beta^{+}\text{, } \alpha\leq \beta.
\]
As usual, we write $\alpha<\beta$ (resp. $\alpha\prec\beta$) if $\alpha\leq \beta$ (resp. $\alpha\preceq\beta$) but $\alpha\neq \beta$. Write $\alpha\not<\beta$ (resp. $\alpha\not\prec\beta$) if $\alpha<\beta$ (resp. $\alpha\prec\beta$) fails.

\subsection{Symmetric group}
Let $\mathfrak{S}_n$ be the symmetric group of $\{1,\ldots,n\}$ and $s_i=(i,i+1)$ be the adjacent transposition. It is known that $\mathfrak{S}_n$ is generated by $\{s_i:i=1,\ldots,n-1\}$. For $\sigma\in\mathfrak{S}_n$, the length $\ell(\sigma)$ of $\sigma$ is defined as the length of a reduced (i.e. shortest) factorization $\sigma=s_{i_1}\cdots s_{i_\ell}$. It is known that $\ell(\sigma)$ is equal to the number of inversions of the permutation $\sigma$. That is,
\[
\ell(\sigma)=|\{(i<j)|\sigma(i)>\sigma(j)\}|.
\]
The action of $\mathfrak{S}_n$ on $\mathbb{Z}^n$ is defined by
\[
\sigma\alpha:=(\alpha_{\sigma(1)},\alpha_{\sigma(2)},\ldots,\alpha_{\sigma(n)}),
\]
where $\sigma\in\mathfrak{S}_n$, $\alpha=(\alpha_1,\alpha_2,\ldots,\alpha_n)\in\mathbb{Z}^n$. The action of $\mathfrak{S}_n$ on a monomial $x^\alpha:=x_1^{\alpha_1}\cdots x_n^{\alpha_n}$ is defined by
\[
\sigma x^\alpha:=x^{\sigma\alpha}=x_1^{\alpha_{\sigma(1)}}\cdots x_n^{\alpha_{\sigma(n)}}.
\]

\subsection{Symmetric functions}
Throughout this paper, we work with the coefficient field $\mathbb{F}=\mathbb{Q}(q,t)$, where typically in this paper, we assume $t=q^c$ for a nonnegative integer $c$. Let $X=\{x_1,x_2,\dots\}$ be an alphabet of countably many variables and $X_n=\{x_1,\ldots,x_n\}$ be a finite alphabet. Denote the ring of symmetric functions in $X$ (resp. $X_n$) over the field $\mathbb{F}$ by $\Lambda_{\mathbb{F}}$ (resp. $\Lambda_{\mathbb{F},n}$). Typical bases for $\Lambda_{\mathbb{F}}$ include the monomial, elementary, complete, and power sum symmetric functions. Here we introduce only the monomial and the power sum symmetric functions. For more details, see~\cite{Mac95}.

For a partition $\lambda$, the monomial symmetric function $m_\lambda$ in the alphabet $X$ is defined as
\[
m_{\lambda}=m_{\lambda}(X):=\sum_{\alpha^+=\lambda} X^\alpha,
\]
where $\alpha$ runs over all distinct rearrangements of $\lambda$ and  $X^\alpha:=x_1^{\alpha_1} x_2^{\alpha_2}\cdots$. Define $m_\lambda(X_n):=m_\lambda(X)\big|_{x_i=0\text{ for }i>n}$. It follows that $m_\lambda(X_n)=0$ if $\ell(\lambda)>n$. 

For a positive integer $r$, define
\[
p_r=p_r(X):=\sum_{i\geq 1}x_i^r,
\]
and $p_0=1$. More generally, for a partition $\lambda$, the power sum symmetric function $p_\lambda$ in $X$ is defined by
\[
p_{\lambda}=p_{\lambda_1}p_{\lambda_2}\cdots.
\]
The $p_r$ are algebraically independent over
$\mathbb{Q}$, and the $p_{\lambda}$ form a basis of $\Lambda_{\mathbb{Q}}$\cite{Mac95}.

\subsection{Plethystic notation and substitution}\label{subsec-2.5}
The plethystic notation is now a fundamental tool when dealing with symmetric functions, see e.g.,~\cite{haglund,Lascoux,RW}. For $f(X)$ a symmetric function in $X$, we additively write $X=x_1+x_2+\cdots$, and use the plethystic brackets to indicate this additive notation:
\[
f(X)=f(x_1,x_2,\ldots)=f[x_1+x_2+\cdots]=f[X].
\]
Then for two alphabets $X$ and $Y$, the sum $X+Y$ is the disjoint union of these sets. 

The above notation implies the following identity for the power sum symmetric functions:
\begin{equation}\label{asm1}
    p_r[X+Y]=p_r[X]+p_r[Y].
\end{equation}
This identity further motivates the following definitions that a power sum symmetric function acts on the difference and Cartesian product of two alphabets.
\begin{subequations}\label{asm}
\begin{align}
p_r[X-Y]&:=p_r[X]-p_r[Y], \\
p_r[XY]&:=p_r[X]p_r[Y].\label{asm3}
\end{align}
\end{subequations}
Suppose $\alpha$ is a single-letter alphabet such that $|\alpha|<1$. By \eqref{asm1} and \eqref{asm3},
\[
p_r\left[X(1+\alpha+\alpha^2+\cdots)\right]=p_r[X]\sum_{k=0}^{\infty} p_r[\alpha^k]=
p_r[X]\sum_{k=0}^{\infty} \alpha^{kr}=\frac{p_r[X]}{1-\alpha^r}.
\]
Hence, it is natural to define the next division rule:
\begin{equation}\label{division}
p_r[X/(1-\alpha)]:=p_r[X(1+\alpha+\alpha^2+\cdots)]=\frac{p_r[X]}{1-\alpha^r}.
\end{equation}
Note that we cannot define the plethystic division for an arbitrary alphabet. Since the $p_r$ are algebraically independent and generate $\Lambda_{\mathbb{Q}}$, these definitions are extended to $f[X-Y]$, $f[XY]$ and $f[X/(1-\alpha)]$ for any symmetric function $f\in\Lambda_{\mathbb{Q}}$.

From \eqref{asm1}, we also have
\[
p_r[2X]:=p_r[X+X]=2p_r[X].
\]
We extend the above identity to any $k\in \mathbb{F}$ by defining
\begin{equation}
p_r[kX]:=kp_r[X].
\end{equation}
Note that this leads to some notational ambiguities, and whenever it is not clear from the context we will indicate if a symbol represents a letter or a binomial element~\cite{Lascoux}.

Occasionally we need to use an ordinary minus sign in plethystic notation.
To distinguish this from a plethystic minus sign, we denote by $\epsilon$ the alphabet consisting of the single letter $-1$, so that for $f\in \Lambda_{\mathbb{F}}$
\[
f(-x)=f(-x_1,-x_2,\dots)=f[\epsilon x_1+\epsilon x_2+\cdots]=f[\epsilon X].
\]
In particular,
\[
p_r[\epsilon X]=(-1)^rp_r[X].
\]
For a single-letter alphabet $a$, if $f$ is a homogeneous symmetric function of degree $k$ then
\begin{equation}\label{e-homo-sym}
f[aX]=f(ax_1,ax_2,\ldots)=a^k f(x_1,x_2,\ldots)=a^kf[X].
\end{equation}
In particular, we have
\begin{equation}\label{e-Mac3}
f[\epsilon X]=(-1)^kf[X].
\end{equation}

\subsection{The symmetric Macdonald polynomials} Define the $(q,t)$-Hall scalar product on $\Lambda_{\mathbb{F}}$ by~\cite[page 306]{Mac95}
\[
\langle p_{\lambda}, p_{\mu}\rangle:=\delta_{\lambda,\mu}z_{\lambda}\prod_{i=1}^{\ell(\lambda)}\frac{1-q^{\lambda_i}}{1-t^{\lambda_i}},
\]
where $\lambda$ and $\mu$ are partitions, $\delta_{\lambda,\mu}$ is the Kronecker symbol, and $z_{\lambda}:=\prod_{i\geq 1}i^{m_i}m_i!$ if $m_i$ is the number of parts of $\lambda$ that equal $i$. The Macdonald polynomials $P_{\lambda}=P_{\lambda}(X;q,t)$ are the unique symmetric functions \cite[VI, (4.7)]{Mac95} such that
\[
\langle P_{\lambda},P_{\mu}\rangle=0 \quad \text{if} \quad \lambda\neq \mu
\]
and
\begin{equation}\label{eq-Mac-1}
P_{\lambda}=m_{\lambda}+\sum_{\mu<\lambda}c_{\lambda\mu}m_{\mu}, \quad c_{\lambda\mu}\in \mathbb{F}.
\end{equation}
Here $\leq$ is the dominance order on partitions. By~\eqref{eq-Mac-1}, $P_\lambda(X_n;q,t)=0$ if $\ell(\lambda)>n$, and the sets $\{P_\lambda(X)\}$ and $\{P_\lambda(X_n)\}_{\ell(\lambda)\leq n}$ form bases for $\Lambda_{\mathbb{F}}$ and $\Lambda_{\mathbb{F},n}$, respectively. From~\cite[VI, (4.14)]{Mac95} we have
\begin{equation}\label{eq-mac-qt-inverse}
P_\lambda(X;q,t)=P_\lambda(X;q^{-1},t^{-1}).
\end{equation}
If $\lambda$ is a partition of length $n$, then~\cite[VI, (4.17)]{Mac95}
\begin{equation}\label{eq-Mac-3}
P_{\lambda}(X_n;q,t)=x_1\cdots x_n P_{(\lambda_1-1,\ldots,\lambda_n-1)}(X_n;q,t).
\end{equation}

In the subsequent proof we need the following vanishing condition for Macdonald polynomials. 
\begin{prop}~\cite[Proposition 3.3]{Zhou-tran}\label{prop-Mac-vanish}
Let $\lambda$ be a partition and let $P_{\lambda}(X;q,t)$ be the Macdonald polynomial. For any nonzero $\lambda_i$, 
\begin{equation}\label{eq-Mac-4}
P_{\lambda}\Big(\Big[\frac{1-q}{t-1}\big(\alpha_1+\cdots+\alpha_{\lambda_i-1}\big)+\beta_1+\cdots+\beta_{i-1}\Big];q,t\Big)=0,
\end{equation}
where the $\alpha_j$ and the $\beta_j$ are single-letter alphabets (or monic monomials).
\end{prop}

\begin{comment}
We also require the skew Macdonald polynomials defined by
\begin{equation}\label{eq-mac-skew}
P_\lambda\big(\big[ X+Y \big];q,t\big)=\sum_\mu P_{\lambda/\mu}\big(X;q,t\big) P_\mu \big(Y;q,t\big),
\end{equation}
where $X$ and $Y$ are two disjoint alphabets. In particular, $P_{\lambda/\emptyset}=P_\lambda$. The skew Macdonald polynomial $P_{\lambda/\mu}$ is homogeneous of degree $|\lambda|-|\mu|$ and vanishes unless $\mu\subseteq \lambda$ \cite[p. 344]{Mac95}
\end{comment}

\subsection{The nonsymmetric Macdonald polynomials}
Define a scalar product on $\mathbb{F}\big[x_1,\ldots,x_n\big]$ by
\begin{equation}\label{def-nonmac-innerproduct}
\langle f, g\rangle_{n}=\langle f, g\rangle_{q,t,n}:=\underset{x}{\mathrm{CT}}\big(f(x ; q, t) g(x^{-1}; q^{-1},t^{-1}) W(x)\big),
\end{equation}
where
\[
W(x)=W(x;q,t):=\prod_{1\leq i<j\leq n}\cfrac{(x_i/x_j)_\infty (qx_j/x_i)_\infty}{(tx_i/x_j)_\infty (qtx_j/x_i)_\infty}
\]
is the weight function. In particular, for $t=q^c$ with $c$ a nonnegative integer,
\[
W(x;q,q^c)=\prod_{1 \leq i<j \leq n}(x_i/x_j )_c (qx_j/x_i)_c.
\]
For the relation between $\langle-,-\rangle_n$ and $\langle-,-\rangle$, we refer the reader to~\cite[\S5]{Mac-affine}. 

The nonsymmetric Macdonald polynomials $E_\nu(x;q,t)\in\mathbb{F}[x_1,\ldots,x_n]$ ($\nu\in\mathbb{N}^n$) are the unique polynomials such that
\begin{equation}\label{eq-nonmac-ortho}
\langle E_\nu(x;q,t),E_\eta(x;q,t)\rangle_n=0\quad \text{ if $\nu\neq \eta$}
\end{equation}
and
\begin{equation}\label{eq-nonmac-uppertriangular}
E_\nu(x;q,t)=x^\nu+\sum_{\eta\prec \nu}c_{\nu\eta} x^\eta,\quad c_{\nu\eta}\in\mathbb{F}.
\end{equation}
Here $\preceq$ is the Bruhat order on compositions. By~\eqref{eq-nonmac-uppertriangular}, the nonsymmetric Macdonald polynomial $E_\nu(x;q,t)$ is homogeneous of degree $|\nu|$, and the set $\{E_\nu(x;q,t):\nu\in\mathbb{N}^n\}$ forms a basis of $\mathbb{F}[x_1,\ldots,x_n]$. 

It is not hard to obtain the following stability property from the definition.
\begin{equation}\label{eq-nonmac-stability}
E_{\nu}(x_1,\ldots,x_{n-1},x_n)\Mid_{x_n=0}=
\begin{cases}
E_{\nu^{(n)}}(x_1,\ldots,x_{n-1})& \text{if }\nu_n = 0,\\
0& \text{if }\nu_n\neq 0,
\end{cases}
\end{equation}
where $\nu^{(n)}=(\nu_1,\ldots,\nu_{n - 1})$. From~\eqref{eq-nonmac-ortho} and ~\eqref{eq-nonmac-uppertriangular}, we can also obtain the following lemma.
\begin{lem}\label{lem-nonmac-vanish-property}
    Let $\eta ,\nu\in\mathbb{N}^n$. Then $\langle x^\eta,E_\nu \rangle_{n}=0$ if $\nu\not\preceq \eta$.
\end{lem}

We also need the next lemma and its proof is similar to the case of symmetric Macdonald polynomials\cite[Section 6.7]{Mac95}. However, due to the lack of a suitable direct reference, we relegate its proof to the Appendix.
\begin{lem}\label{cor-nonmac-splitvariables}
   For an integer $s$ such that $0\leq s\leq n$, let $I:=\{i_1,\ldots,i_s\}$ and $J:=\{j_1,\ldots,j_{n-s}\}$ be two disjoint sets such that $I\cup J=\{1,\ldots,n\}$. The nonsymmetric Macdonald polynomials $E_\nu(x;q^{-1},t^{-1})$ can be expanded as
    \begin{equation}\label{eq-nonmac-splitvariables-finalform} E_\nu(x;q^{-1},t^{-1})=\sum_{\substack{\rho\in\mathbb{N}^s,\gamma\in \mathbb{N}^{n-s}\\\rho^+,\gamma^+\subseteq \nu^+}}
    c_{\rho\gamma}^\nu(I;J) E_{\rho}(x_{i_1},\ldots,x_{i_s};q^{-1},t^{-1}) E_\gamma(x_{j_1},\ldots,x_{j_{n-s}};q^{-1},t^{-1}),
    \end{equation}
    where $c_{\rho\gamma}^\nu(I;J)\in \mathbb{F}$.
\end{lem}

We also need a simple fact that follows from Lemma~\ref{cor-nonmac-splitvariables}.
\begin{cor}\label{cor-nonmac-degree}
    Let  $\{u_1,\ldots,u_s\}$ be a subset of $\{1,\ldots,n\}$, $\nu=(\nu_1,\ldots,\nu_n)$ be a composition and $\nu^+=(\nu_1^+,\ldots,\nu_n^+)$ be the partition obtained by reordering the entries of $\nu$. Then the total degree of $E_{\nu}(x;q,t)$ in the variables $x_{u_1},\ldots,x_{u_s}$ is at most $\nu_1^++\cdots+\nu_{s}^+$ and at least $\nu_{n-s+1}^+ +\cdots+\nu_n^+$.
\end{cor}

\subsection{Macdonald polynomials with prescribed symmetry}\label{subsec-2-8}
Define operators $T_i$ acting on the Laurent polynomial ring $\mathbb{F}[x_1^\pm,\ldots,x_n^\pm]$ by
\begin{equation}\label{definition-nonmac-T_i}
T_i:=t+\frac{tx_i - x_{i + 1}}{x_i - x_{i+1}}(s_i-1),\quad i = 1,\ldots,n - 1.
\end{equation}
The operators $T_i$ are known as the Demazure--Lusztig operators. They satisfy the braid relations
\begin{subequations}\begin{equation}\label{eq-nonmac-hecke-relation-2}T_iT_{i + 1}T_i=T_{i + 1}T_iT_{i + 1}  \quad \text{ for $i=1,\ldots,n-2$}\end{equation}\begin{equation}\label{eq-nonmac-hecke-relation-3}T_iT_j=T_jT_i, \quad \text{ for $|i-j|>1$},\end{equation}\end{subequations}
and the quadratic relations
\label{eq-nonmac-hecke-relation}\begin{equation}\label{eq-nonmac-hecke-relation-1}(T_i - t)(T_i+1)=0 \quad \text{ for $i=1,\ldots,n-1$}.\end{equation}
Note that the operators $T_i$ reduce to the transpositions $s_i$ when $t=1$. For $f\in\mathbb{F}[x_1,\ldots,x_n]$, $T_if=t f$ if and only if $f$ is symmetric in $x_i$ and $x_{i+1}$. 

According to \cite{cauchy-formula}, for $i=1,\ldots,n-1$, the operator $T_i$ acts on $E_{\nu}$ as
\begin{equation}\label{eq-nonmac-T_iaction}
T_iE_{\nu}=\begin{cases}
     c_{i,\nu}E_\nu+c_{i,\nu}' E_{s_i\nu}\quad & \text{ if $\nu_i\neq \nu_{i+1}$},\\
     t E_\nu & \text{ if $\nu_i= \nu_{i+1}$},
\end{cases}
\end{equation}
where $c_{i,\nu}\neq c_{i,\nu}'\in\mathbb{F}\setminus\{0\}$ are explicitly given in\cite[(1.19)-(1.22)]{cauchy-formula}. Therefore, $E_\nu(x;q,t)$ is symmetric in $x_i$ and $x_{i+1}$ if and only if $\nu_i=\nu_{i+1}$.

For $\sigma\in\mathfrak{S}_n$ with reduced factorization $\sigma=s_{i_1}\cdots s_{i_\ell}$, define $T_\sigma:=T_{i_1}\cdots T_{i_\ell}$. By the braid relations for the $T_i$, such a definition does not depend on the choice of reduced expression of $\sigma$. In what follows, for $f\in\mathbb{F}[x_1,\ldots,x_n]$, we say $f$ is $t$-symmetric (resp. $t$-antisymmetric) in $x_1,\ldots,x_n$ if for any $\sigma\in\mathfrak{S}_n$, we have $T_\sigma f=t^{\ell(\sigma)}f$ (resp. $T_\sigma f=(-1)^{\ell(\sigma)}f$).

The $t$-symmetrization and $t$-antisymmetrization operators on the variables $X_n=\{x_1,\ldots,x_n\}$ are then defined as~\cite{Mac-affine}
\begin{equation}\label{definition-premac-sym,antisym-operator}
\mathcal{S}_t=\mathcal{S}_t^{(X_n)}:=\sum_{\sigma\in \mathfrak{S}_n} T_\sigma \quad \text{ and } \quad \mathcal{A}_t=\mathcal{A}_t^{(X_n)}:=\sum_{\sigma\in\mathfrak{S}_n} (-t)^{-\ell(\sigma)}T_\sigma.
\end{equation}
It is routine to check that~\cite{BDF}
\begin{equation*}
\begin{aligned}
 &\mathcal{S}_tT_{i}^{\pm 1}=T_{i}^{\pm 1}\mathcal{S}_t=t^{\pm 1}\mathcal{S}_t,\\
&\mathcal{A}_tT_{i}^{\pm 1}=T_{i}^{\pm 1}\mathcal{A}_t=-\mathcal{A}_t.
\end{aligned}
\end{equation*}
Therefore, any Laurent polynomial $f$ under $t$-symmetrization becomes $t$-symmetric, and becomes $t$-antisymmetric under $t$-antisymmetrization. In particular, for the staircase partition $\delta=\delta^n:=(n - 1,\ldots,2,1,0)$, we have
\[
t^{n(n - 1)/2}\mathcal{A}_tx^{\delta}=\prod_{1\leq i<j\leq n}(tx_{i}-x_{j})=:\Delta_{t}(X_n),
\]
and
\[
\mathcal{A}_tx^{\gamma}=0\quad\text{if } \gamma_{i}=\gamma_{j}\text{ for some }i\neq j.
\]
The notation $\Delta_t(X_n)$ will be referred to as the $t$-Vandermonde determinant.

We can express $\mathcal{S}_t$ and $\mathcal{A}_t$ in terms of the usual antisymmetrization operator $\mathcal{A}:=\mathcal{A}_t\Mid_{t=1}=\sum_{\sigma\in\mathfrak{S}_n}(-1)^{-\ell(\sigma)}\sigma$ by~\cite[eq.(2.26)] {Marshall}
\[
\mathcal{S}_t f(X_n)=t^{n(n - 1)/2}\cfrac{\mathcal{A}(\Delta_{t^{-1}}(X_n)f(X_n))}{\Delta(X_n)}\quad \text{ and }\quad \mathcal{A}_t f(X_n)=t^{-n(n - 1)/2}\cfrac{\Delta_t(X_n)}{\Delta(X_n)}\mathcal{A}f(X_n),
\]
where $\Delta(X_n):=\Delta_t(X_n)\Mid_{t=1}=\prod_{1\leq i<j\leq n}(x_i-x_j)$ is the usual Vandermonde determinant. Since $\mathcal{A}f(X_n)$ is anti-symmetric in $X_n$, it is divisible by the denominator $\Delta(X_n)$. Therefore $\mathcal{A}_t f(X_n)$ always has the factor $\Delta_t(X_n)$.

It is natural to restrict all the discussion to a subset of $X_n$. Specifically, let 
\[
\mathfrak{S}_m:=\langle s_i:1\leq i\leq m-1\rangle\quad \text{ and }\quad \mathfrak{S}_{n\setminus m}:=\langle s_i:m+1\leq i\leq n-1\rangle
\]
be the subgroups of $\mathfrak{S}_n$ acting on the variables $X_m=\{x_1,\ldots,x_m\}$ and $X_{n\setminus m}=\{x_{m+1},\ldots,x_{n}\}$, respectively. Define the $t$-antisymmetrization operator and $t$-symmetrization operator on $X_m$ and $X_{n\setminus m}$, respectively, as
\begin{equation}\label{def-nonmac-prescibed-partial-sym-antisym}
\mathcal{A}_t^{(X_m)}:=\sum_{\sigma\in \mathfrak{S}_m} (-t)^{-\ell(\sigma)} T_\sigma \quad \text{ and } \quad \mathcal{S}_t^{(X_{n\setminus m})}:=\sum_{\sigma\in \mathfrak{S}_{n\setminus m}}   T_\sigma.
\end{equation}
Then for any Laurent polynomial $f$ in $X_n$, the Laurent polynomial $\mathcal{A}_t^{(X_m)}\mathcal{S}_t^{(X_{n\setminus m})} f$ is $t$-antisymmetric in the first $m$ variables $X_m$, and $t$-symmetric in the remaining variables $X_{n\setminus m}$.

\begin{rem}\label{rem-nonmac}
    It follows from (\ref{eq-nonmac-T_iaction}) and (\ref{def-nonmac-prescibed-partial-sym-antisym}) that $\mathcal{A}_t^{(X_m)}\mathcal{S}_t^{(X_{n\setminus m})} E_\eta(x;q,t)$ is actually a linear combination of those $E_\gamma(x;q,t)$ with $\gamma^+=\eta^+$.
\end{rem}

The following theorem reveals that under a certain restriction on the composition $\eta$, the polynomial $\mathcal{A}_t^{(X_m)}\mathcal{S}_t^{(X_{n\setminus m})} E_\eta(x;q,t)$ is either zero (if $\eta_i=\eta_{i+1}$ for some $1\leq i\leq m-1$), or, up to a constant multiple, a product of a $t$-Vandermonde determinant and symmetric Macdonald polynomials.
\begin{theorem}\label{thm-nonmac-factorization}
    Let $n,m$ be positive integers such that $m\leq n$, $\lambda,\xi$ be partitions such that $\ell(\lambda)< m$, $\ell(\xi)\leq n-m$, and $\eta:=(\lambda+\delta^m,\xi)$. Then we can write
    \begin{equation}
  \mathcal{A}_t^{(X_m)}\mathcal{S}_t^{(X_{n\setminus m})} E_\eta(x;q,t)= c_\eta \Delta_t(X_m) P_\lambda\Big(\Big[X_m+\cfrac{q(1-t)}{1-qt}X_{n\setminus m}\Big];q,qt\Big) P_\xi\Big(\Big[ X_{n\setminus m} \Big];qt,t\Big)
    \end{equation}
    in the following cases:
    \begin{enumerate}
        \item $\xi=\emptyset$;
        \item $\lambda=\emptyset$ and $\xi_1\leq m$;
        \item $\lambda,\xi\neq \emptyset$, and $|\lambda|+|\xi|\leq \min\{n-m,m\}$.
    \end{enumerate}
    Here $c_\eta\in\mathbb{F}$ is a nonzero constant.
\end{theorem}
Note that the conditions in the above theorem are the same as those in Theorem~\ref{main-thm}. The second case of the theorem was established by Baker et al. \cite[Proposition 2]{BDF}, while the third case was found by Blondeau-Fournier et al. \cite[Theorem 3]{double-mac}. However, to our knowledge, the first case of the above theorem has not been previously established. We will give its proof in Section~\ref{sec-proof4}. We also note that the constant $c_\eta$ in the above theorem in each case can be explicitly computed, but we omit its precise expression as it is not necessary in this paper.

\subsection{Constant term evaluations}
In this subsection, we present a basic lemma for extracting constant terms from rational functions in the framework of the field of iterated Laurent series. For more details, we refer to~\cite{xinresidue} and~\cite{xiniterate}.

Throughout this paper, we work in the field of iterated
Laurent series 
\[
\mathbb{F}\langle\!\langle x_n, x_{n-1},\dots,x_0\rangle\!\rangle
=\mathbb{F}(\!(x_n)\!)(\!(x_{n-1})\!)\cdots (\!(x_0)\!).
\]
Elements of $\mathbb{F}\langle\!\langle x_n,x_{n-1},\dots,x_0\rangle\!\rangle$
are regarded first as Laurent series in $x_0$, then as
Laurent series in $x_1$, and so on. 
The most applicable fact is that the field $\mathbb{F}(x_0,\dots,x_n)$ of
rational functions in the variables $x_0,\dots,x_n$ with coefficients in $\mathbb{F}$
forms a subfield of $\mathbb{F}\langle\!\langle x_n, x_{n-1},\dots,x_0\rangle\!\rangle$, so that every rational function can be identified with its unique Laurent series expansion.

For $F(x) \in  \mathbb{F}\langle\!\langle x_n, x_{n-1},\dots,x_0\rangle\!\rangle$, the constant term of $F(x)$ in $x_i$, denoted by
$\underset{x_i}{\CT} F(x)$, is defined to be the sum of those terms
in the series expansion of $F(x)$ that are free of $x_i$. The constant term operators defined in $\mathbb{F}\langle\!\langle x_n,\dots,x_0\rangle\!\rangle$ commute with each other, i.e.
\[
\CT_{x_i} \CT _{x_j} f = \CT_{x_j} \CT_{x_i} f
\]
for any $i$ and $j$. Commutativity implies that taking the constant term with respect to a set of variables is well-defined. Therefore, for variables $x=\{x_1,\ldots,x_n\}$ we define
\[
\CT\limits_x:=\CT\limits_{x_1}\cdots \CT\limits_{x_n}
\]
in $\mathbb{F}\langle\!\langle x_n,\dots,x_0\rangle\!\rangle$. This definition extends the constant term operators used earlier.

In $\mathbb{F}\langle\!\langle x_n,\dots,x_0\rangle\!\rangle$ we might formally assume $0\ll x_0\ll x_1\ll \cdots \ll x_n\ll 1$. Here $x\ll y$ means $x$ is much less than $y$. That is, $|cx/y|<1$ for any $c$ belonging to a finite subset of $\mathbb{F}$. Hence, we can write
\[
\frac{1}{1-c x_i/x_j}=
\begin{cases} \displaystyle \sum_{l\geq 0} c^{l} (x_i/x_j)^l
& \text{if $i<j$}, \\[5mm]
\displaystyle -\sum_{l<0} c^{l} (x_i/x_j)^l
& \text{if $i>j$}
\end{cases}
\]
in this paper for $c\in\mathbb{F}\setminus\{0\}$. This implies that
\begin{equation}\label{eq-preliminary-laurent-ct}
\CT_{x_i} \frac{1}{1-c x_i/x_j} =
\begin{cases}
    1 & \text{if $i<j$}, \\
    0 & \text{if $i>j$}. \\
\end{cases}
\end{equation}

The following lemma, which first appeared in~\cite{GX}, is a basic tool for
extracting constant terms from rational functions.
\begin{lem}\cite[Lemma 4.1]{GX}\label{lem-preliminary-laruent-proper}
For a positive integer $l$, let $p(x_k)$ be a Laurent polynomial
in $x_k$ of degree at most $l-1$ with coefficients in
$\mathbb{F}\langle\!\langle x_n,\dots,x_{k+1},x_{k-1},\dots,x_0\rangle\!\rangle$.
Consider indices $0\leq i_1\leq\dots\leq i_l\leq n$ such that $i_r\neq k$ for $r=1,\ldots,l$. Then for
\begin{equation}\label{eq-preliminary-laruent-proper-1}
f=\frac{p(x_k)}{\prod_{r=1}^l (1-c_r x_k/x_{i_r})}
\end{equation}
with $c_1,\dots,c_l\in \mathbb{F}\setminus \{0\}$ such that $c_r\neq c_s$ whenever $x_{i_r}=x_{i_s}$, we have
\begin{equation}\label{eq-preliminary-laruent-proper-2}
\CT_{x_k} f=\sum_{\substack{r=1 \\[1pt] i_r>k}}^l
\big(f\,(1-c_rx_k/x_{i_r})\big)\Big|_{x_k=c_r^{-1}x_{i_r}}.
\end{equation}
\end{lem}

We also need the following two propositions. Each of them provides an equivalence between two kinds of constant terms. 
\begin{prop}\cite[Proposition 4.2]{Zhou-AFLT}\label{prop-preliminary-laruent-equiv}
Let $f(x)$ be a Laurent polynomial that is invariant under any permutation of $x$. Then
\begin{equation}\label{prop-preliminary-laruent-equiv-1}
\CT_{x} f(x)\prod_{1\leq i<j\leq n}(x_i/x_j)_c(qx_j/x_i)_c
=\frac{1}{n!}\prod_{i=1}^{n-1}\frac{1-q^{(i+1)c}}{1-q^c}\times\CT_{x} f(x)\prod_{1\leq i\neq j\leq n} (x_i/x_j)_c.
\end{equation}
\end{prop}

\begin{prop}\cite[Lemma 9]{baratta}\label{prop-preliminary-antisymmetric}
    Let $J=\{r,\ldots,r+s\}\subseteq \{1,\ldots,n\}$ and $h(x)$ be a Laurent polynomial that is antisymmetric with respect to $\{x_j,j\in J\}$. Then, for $a\in\mathbb{F}\setminus\{0\}$, 
    \[
    \CT_x \left( \prod_{r\leq i<j\leq r+s}(x_i-ax_j)h(x)\right)=\cfrac{[s+1]_a!}{(s+1)!}\CT_x \left( \prod_{r\leq i<j\leq r+s}(x_i-x_j)h(x)\right),
    \]
    where $[s+1]_a!=[1]_a\cdots [s+1]_a:=\cfrac{(a;a)_{s+1}}{(1-a)^{s+1}}$.
\end{prop}

\subsection{Polynomiality and Rationality}
A crucial step in the proof of Theorem~\ref{main-thm} is to treat constant terms as polynomials or rational functions. We first present the polynomiality result. We omit its proof as it is almost the same as that of~\cite[Lemma 2.2]{XZ}.

\begin{lem}[Polynomiality]\label{lem-preliminary-ct-polynomiality}
Let $L(x_1,\dots,x_n)$ be an arbitrary homogeneous Laurent polynomial that is independent of $a$ (resp. $b$) and $x_0$ with degree $-\tau$. For a nonnegative integer $b$  (resp. $a$) such that $nb-\tau\geq 0$ (resp. $na+\tau\geq 0$), the constant term 
\begin{equation}\label{eq-preliminary-ct-polynomiality}
\CT_x x_0^{\tau} L(x_1,\dots,x_n) \prod_{i=1}^n (x_0/x_i)_a(qx_i/x_0)_b
\end{equation}
is a polynomial in $q^a$ (resp. $q^b$) of degree at most $nb-\tau$ (resp. $na+\tau$). Note that when $nb-\tau<0$ (resp. $na+\tau<0$), the constant term above trivially vanishes by considering the degree of $x_0$. 
\end{lem}

The following rationality result is implicitly due to Stembridge \cite{stembridge1987}, as it can be seen from the proof. One can also see this result in \cite[Proposition~3.1]{XZ} and \cite[Lemma~7.5]{KNPV}. 
\begin{prop}[Rationality]\label{prop-preliminary-stembridge-rationality}
Let $\alpha=(\alpha_1,\dots,\alpha_n)\in \mathbb{Z}^{n}$ be such that
$|\alpha|=0$. Then, for a nonnegative integer $c$, 
\begin{align}
\CT_x x_1^{\alpha_1}\cdots  x_n^{\alpha_n}\prod_{1\leq i<j\leq n}
(x_i/x_j)_c(qx_j/x_i)_c=\frac{(q)_{nc}}{(q)_{c}^{n}}\cdot
R_n(q^c,q;\alpha),
\end{align}
where $R_n(q^c,q;\alpha)$ is a rational function in $q^c$ and $q$.
\end{prop}

By the next lemma, it suffices to prove Theorem~\ref{main-thm} under the restriction that $c\geq b+\mu_1$ and $b\geq |\lambda|+|\xi|+m$.
\begin{lem}\label{lem-preliminary-a,b,c}
    Let $n,k,l$ be integers, and let $C_1(a,b,c)$ and $C_2(a,b,c)$ be two expressions in the integer parameters $a,b,c$, such that
    \begin{enumerate}[(a)]
        \item for fixed nonnegative $b$ and $c$ such that $b\geq l$,  both $C_1$ and $C_2$ are polynomials in $q^a$ of degree at most $nb+k$;
        \item for fixed nonnegative $a$ and $c$, both $C_1$ and $C_2$ are rational functions in $q^b$;
        \item for fixed nonnegative $a$ and $b$, both $C_1$ and $C_2$ are rational functions in $q^c$.
    \end{enumerate}
    Let $t_1,t_2$ be two nonnegative integers that are independent of $a,b,c$, and $\mathcal{S}\subset \mathbb{Z}$ be a set of $nb+k+1$ distinct integers (elements in $\mathcal{S}$ might depend on $b$ and $c$, but are independent of $a$). If 
    \begin{equation}\label{eq-preliminary-abc-C}
    C_1(a,b,c)=C_2(a,b,c)
    \end{equation}
    for all triples $a,b,c$ such that $c\geq b+t_1, b\geq \max\{t_2,l\}, a\in\mathcal{S}$, then it holds for all nonnegative integers $a,b,c$.
\end{lem}
\begin{proof}
For fixed nonnegative integers $b$ and $c$ such that $c\geq b+t_1, b\geq \max\{t_2,l\}$, both $C_1(a,b,c)$ and $C_2(a,b,c)$ are polynomials in $q^a$ of degree at most $nb+k$. Since they coincide at $nb+k+1$ distinct points, they are identical as polynomials. Hence the equality~\eqref{eq-preliminary-abc-C} holds for all triples $(a,b,c)$ such that $a\in\mathbb{N}, b\geq \max\{t_2,l\}, c\geq b+t_1$.

Now, for any fixed $a\in\mathbb{N}$ and $b\geq \max\{t_2,l\}$, both $C_1(a,b,c)$ and $C_2(a,b,c)$ are rational functions in $q^c$. They agree for infinitely many values $c\geq b+t_1$. Hence, they are identical as rational functions in $q^c$. That is, the equality~\eqref{eq-preliminary-abc-C} holds for all triples $(a,b,c)$ such that $a,c\in\mathbb{N}$, $b\geq \max\{t_2,l\}$.

Finally, for any fixed $a,c\in\mathbb{N}$, both $C_1(a,b,c)$ and $C_2(a,b,c)$ are rational functions in $q^b$. By the same argument as above, the equality \eqref{eq-preliminary-abc-C} holds for all triples $(a,b,c)$ such that $a,b,c\in\mathbb{N}$.
\end{proof}

%% file: 3.Twofamilies.tex
\section{Two families of constant terms}\label{section-two families}
In this section we introduce two families of constant terms, which are intimately related to the constant term $F_{n,m}(a,b,c,\lambda,\xi,\mu)$. Their polynomiality, rationality and vanishing properties lead to Step (1) and Step (2) in the introduction.

\subsection{The first family} 
For a nonnegative integer $\tau$, let
$H_\tau=H_\tau(x)$ be a homogeneous Laurent polynomial of the form
\begin{equation}\label{def-intro-H_tau}
H_\tau(x)=\sum_{\substack{\alpha\in\mathbb{N}^n\\ |\alpha|=\tau}} c_\alpha x_1^{-\alpha_1}\cdots x_n^{-\alpha_n},\quad c_\alpha\in\mathbb{F}.
\end{equation}
For nonnegative integers $a,b,c$ and a partition $\mu$ such that $\ell(\mu)<m$, define
\begin{multline}\label{def-intro-F(mu,H)}
    F_{n,m}(a,b,c,H_\tau,\mu):=\CT_x x_0^{\tau-|\mu|} H_\tau F_{n,m}(x;a,b,c)\\
    \times P_\mu\left(\left[\cfrac{q^{c-b}-q^a}{1-q^{c+1}}x_0+X_m +\cfrac{q-q^{c+1}}{1-q^{c+1}}X_{n\setminus m}\right];q,q^{c+1}\right).
\end{multline}
The constant term $F_{n,m}(a,b,c,\lambda,\xi,\mu)$ is then a special case of $F_{n,m}(a,b,c,H_\tau,\mu)$ by setting $\tau=|\lambda|+|\xi|$ and 
\begin{equation}\label{eq-sec2-specific-H-tau}
H_\tau=P_\lambda\left(\left[X_m^{-1}+\cfrac{1-q^c}{1-q^{c+1}}X_{n\setminus m}^{-1}\right];q,q^{c+1} \right)P_\xi\left( X_{n\setminus m}^{-1};q^{c+1},q^c\right).
\end{equation}

The constant term $F_{n,m}(a,b,c,H_\tau,\mu)$ has the properties of polynomiality, rationality, and vanishing.
\begin{customthm}{3.1A}[Polynomiality and rationality]\label{proposition-main-2-1}
    Let $F_{n,m}(a,b,c,H_\tau,\mu)$ be defined in~\eqref{def-intro-F(mu,H)}. If $H_\tau$ is independent of $a$, $b$ and $x_0$, then 
    \begin{enumerate}[(a)]
        \item for fixed nonnegative integers $b$ and $c$, the constant term $F_{n,m}(a,b,c,H_\tau,\mu)$ is a polynomial in $q^a$ of degree at most $nb+|\mu|-\tau$;
        \item for fixed nonnegative integers $a$ and $c$, the constant term $F_{n,m}(a,b,c,H_\tau,\mu)$ is a rational function in $q^b$ ;
        \item for fixed nonnegative integers $a$ and $b$, the expression
        \[
        (q)_c^n/(q)_{nc} \cdot F_{n,m}(a,b,c,H_\tau,\mu)
        \]
        is a rational function in $q^c$.
    \end{enumerate}
\end{customthm}

\begin{customthm}{3.1B}[Vanishing property]\label{proposition-main-2-2}
    If $b$ and $c$ are fixed nonnegative integers such that $c\geq b+\mu_1$ and $b\geq \tau+m$, then as a polynomial in $q^a$,
    \[
    F_{n,m}(a,b,c,H_\tau,\mu)=0 \quad \text{ for $-a\in B_2(\delta^m)\cup B_3(\mu)$}.
    \]
    Here the sets $B_2$ and $B_3$ are defined in~\eqref{def-sets-B2} and ~\eqref{def-sets-B3}, respectively.
\end{customthm}

The proofs of Theorem~\ref{proposition-main-2-1} and Theorem~\ref{proposition-main-2-2} will be given in Section~\ref{sec-proof2}. Consequently, if $H_\tau$ is chosen as in~\eqref{eq-sec2-specific-H-tau}, then Theorem~\ref{proposition-main-2-1} gives the polynomiality and rationality of $F_{n,m}(a,b,c,\lambda,\xi,\mu)$. This establishes Step (1) in the introduction. Moreover, under the assumption that $c\geq b+\mu_1$ and $b\geq |\lambda|+|\xi|+m$, Theorem~\ref{proposition-main-2-2} implies that 
\[
F_{n,m}(a,b,c,\lambda,\xi,\mu)=0 \quad \text{ for $-a\in B_2(\delta^m)\cup B_3(\mu)$}.
\]
This establishes part of Step (2) in the introduction. To complete Step (2), it remains to verify that $F_{n,m}(a,b,c,\lambda,\xi,\mu)=0$ for $-a\in B_1((\lambda+\delta^m,\xi)^+)$. 

\subsection{The second family}
For nonnegative integers $a,b,c$ and compositions $\alpha=(\alpha_1,\ldots,\alpha_n)$ and $\nu=(\nu_1,\ldots,\nu_n)$, define
\begin{multline}\label{definition-A_n(a,b,c,eta,M_tau)}
    M_n(a,b,c,\nu,\alpha):=\CT_x x_0^{|\nu|-|\alpha|} x_1^{\alpha_1}\cdots x_n^{\alpha_n} E_\nu(x^{-1};q^{-1},q^{-c})\\ \times \prod_{i=1}^{n}(x_0/x_i)_a(qx_i/x_0)_b
    \prod_{1\leq i<j\leq n}(x_i/x_j)_c(qx_j/x_i)_c.
\end{multline}
The constant term $M_n(a,b,c,\nu,\alpha)$ also has the properties of polynomiality and vanishing.
\begin{customthm}{3.2}\label{Proposition-main}
For fixed nonnegative integers $b$ and $c$, the constant term $M_n(a,b,c,\nu,\alpha)$ is a polynomial in $q^a$ of degree at most $nb-|\nu|+|\alpha|$. Furthermore, if $c\geq b \geq \nu^+_1$, then
\[
M_n(a,b,c,\nu,\alpha)=0 \quad \text{ for $-a\in B_1(\nu^+)$}.
\]
Here the set $B_1$ is defined in~\eqref{def-sets-B1}, with $\eta$ replaced by $\nu$.
\end{customthm}

Note that the polynomiality of $M_n(a,b,c,\nu,\alpha)$ follows immediately from Lemma~\ref{lem-preliminary-ct-polynomiality}, while the vanishing property described in Theorem~\ref{Proposition-main} will be proved in Section~\ref{sec-proof1}. For $\alpha=\mathbf{0}$ (the zero composition) and $c\geq b\geq \nu_1^+$, as a polynomial in $q^a$, all the roots of $M_n(a,b,c,\nu,\alpha)$ have been identified.

Using Remark~\ref{rem-nonmac} and Theorem~\ref{thm-nonmac-factorization}, under the same conditions on $\lambda$ and $\xi$ as in Theorem~\ref{main-thm}, we can express $F_{n,m}(a,b,c,\lambda,\xi,\mu)$ as a finite sum of certain $M_n(a,b,c,\nu,\alpha)$. To see this, first rewrite $F_{n,m}(a,b,c,\lambda,\xi,\mu)$ as
\begin{equation}\label{eq-F-to-M_n}
    \CT_x x_0^{|\lambda|+|\xi|-|\mu|} S_1\times S_2 \times \prod_{i=1}^n (x_0/x_i)_a (qx_i/x_0)_b \prod_{1\leq i<j \leq n} (x_i/x_j)_c (qx_j/x_i)_c,
\end{equation}
where
\[
S_1=\prod_{1\leq i<j\leq m}(q^{-c} x_i^{-1}-x_j^{-1})P_\lambda\left(\left[X_m^{-1}+\cfrac{1-q^c}{1-q^{c+1}}X_{n\setminus m}^{-1}\right];q,q^{c+1} \right) P_\xi\left(X_{n\setminus m}^{-1};q^{c+1},q^{c}\right)
\]
and
\[
S_2=\prod_{1\leq i<j\leq m}q^c(x_i-q^{c+1}x_j) P_\mu\left(\left[\cfrac{q^{c-b}-q^a}{1-q^{c+1}}x_0+X_m+\cfrac{q-q^{c+1}}{1-q^{c+1}}X_{n\setminus m}\right]; q,q^{c+1}\right).
\]
By~\eqref{eq-mac-qt-inverse}, we have
\[
S_1=\prod_{1\leq i<j\leq m}(q^{-c} x_i^{-1}-x_j^{-1})P_\lambda\left(\left[X_m^{-1}+\cfrac{1-q^c}{1-q^{c+1}}X_{n\setminus m}^{-1}\right];q^{-1},q^{-c-1} \right) P_\xi\left(X_{n\setminus m}^{-1};q^{-c-1},q^{-c}\right).
\]
By Remark~\ref{rem-nonmac} and Theorem~\ref{thm-nonmac-factorization} with $x\mapsto x^{-1},q\mapsto q^{-1}$, we know that $S_1$ is a linear combination of the nonsymmetric Macdonald polynomials $E_\nu(x^{-1};q^{-1},q^{-c})$, where $\nu$ ranges only over the rearrangements of $(\lambda+\delta^m,\xi)$. By replacing $S_1$ with this linear combination and expanding $S_2$ into monomials in~\eqref{eq-F-to-M_n}, we obtain
\begin{equation}\label{eq-sec2-F-to-M_n}
    F_{n,m}(a,b,c,\lambda,\xi,\mu)
    =\sum_{\nu^+=(\lambda+\delta^m,\xi)^+}\sum_{\substack{\alpha\in\mathbb{N}^n\\ |\alpha|\leq |\mu|+\binom{m}{2}}} t_\alpha k_\nu  M_n(a,b,c,\nu,\alpha),
\end{equation}
where $t_\alpha,k_\nu\in\mathbb{F}$. By Theorem~\ref{Proposition-main}, if $c\geq b\geq \lambda_1+m$ and $-a\in B_1(\nu^+)=B_1((\lambda+\delta^m,\xi)^+)$, then every such $M_n(a,b,c,\nu,\alpha)$ in~\eqref{eq-sec2-F-to-M_n} vanishes. It follows that if $c\geq b\geq \lambda_1+m$, then
\[
F_{n,m}(a,b,c,\lambda,\xi,\mu)=0 \quad \text{ for $-a\in B_1((\lambda+\delta^m,\xi)^+)$}.
\]
This completes the remaining part of Step (2) in the introduction.

In conclusion, to establish Step (1) and Step (2) in the introduction, it suffices to prove Theorem~\ref{proposition-main-2-1}, Theorem~\ref{proposition-main-2-2}, Theorem~\ref{Proposition-main}, and the case (1) of Theorem~\ref{thm-nonmac-factorization}.

%% file: 4.Somediscussion.tex
\section{Preparations for Theorem~\ref{proposition-main-2-2} and Theorem~\ref{Proposition-main}}\label{sec-discussion}
In this section we provide some essential results for proving Theorem~\ref{proposition-main-2-2} and Theorem~\ref{Proposition-main}. 

\subsection{Distribution of integers}
The next lemma is similar to~\cite[Lemma 6.4]{Zhou-tran}. Both lemmas describe certain distributions of integers.  Here we prove the lemma using only inequalities, while the proof of~\cite[Lemma 6.4]{Zhou-tran} is based on a weighted tournament on a complete graph. 

\begin{lem}\label{lem-preliminary-key}
For positive integers $s$ and $c$, let $b,e$ and $k_1,\dots,k_s$ be nonnegative integers such that
$1\leq k_{i}\leq (s-1)c+b+e$ for $i=1,\ldots,s$.
Then for an integer $r$ such that $0\leq r\leq s$, at least one of the following holds:
\begin{enumerate}
\item[(i)] $1\leq k_i\leq b$ for some $i$ with $1\leq i\leq s$;
\item[(ii)] $-c-1\leq k_i-k_j\leq c$ for some $(i,j)$ such that $1\leq i<j\leq r$;
\item[(iii)] $-c\leq k_{i}-k_{j}\leq c-1$ for some $(i,j)$ such that $1\leq i<j\leq s$ and $j>r$;
\item[(iv)] there exists a permutation $\omega\in\mathfrak{S}_s$ and nonnegative integers $d_1,\dots,d_s$
such that
\begin{subequations}\label{eq-preliminary-keylemma-1}
\begin{equation}\label{eq-preliminary-keylemma-1-1}
k_{\omega(1)}=b+d_1,
\end{equation}
and
\begin{equation}\label{eq-preliminary-keylemma-1-2}
k_{\omega(j)}-k_{\omega(j-1)}=c+d_j+\chi(1\leq \omega(j)\leq r)\chi(1\leq \omega(j-1)\leq r) \quad \text{for $2\leq j\leq s$.}
\end{equation}
\end{subequations}
Here the $d_j$ satisfy
\begin{equation}\label{eq-preliminary-keylemma-1-t}
r\leq \sum_{j=1}^{s}\big(d_j+\chi(1\leq \omega(j)\leq r)\chi(1\leq \omega(j-1)\leq r)\big)\leq e,
\end{equation}
and $d_j>0$ if $\omega(j-1)<\omega(j)$ for $1\leq j\leq s$, where $\omega(0):=0$.
\end{enumerate}
\end{lem}
\begin{proof}
    Assume that all of (i)-(iii) fail. We show that (iv) must hold. In this case, the $k_i$ are pairwise distinct. Hence, there exists a unique permutation $\omega\in\mathfrak{S}_s$, such that
    \[
    k_{\omega(1)}<k_{\omega(2)}<\cdots<k_{\omega(s)}.
    \]
    Since (i) fails, $k_{\omega(1)}\geq b+1$. Let $d_1:=k_{\omega(1)}-b\geq 1$, then~\eqref{eq-preliminary-keylemma-1-1} holds. Since both (ii) and (iii) fail, we have
    \[
    k_{\omega(j)}-k_{\omega(j-1)}\geq c+\chi(1\leq \omega(j)\leq r)\chi(1\leq \omega(j-1)\leq r)+\chi(\omega(j-1)<\omega(j))
    \]
    for $j=2,\ldots,s$. Let 
    \[
    d_j:=k_{\omega(j)}-k_{\omega(j-1)}-c-\chi(1\leq \omega(j)\leq r)\chi(1\leq \omega(j-1)\leq r).
    \]
    Then~\eqref{eq-preliminary-keylemma-1-2} holds, and the $d_j>0$ if $\omega(j-1)<\omega(j)$. In the following, we show that the $d_j$ satisfy~\eqref{eq-preliminary-keylemma-1-t}. Set $k_0:=0,\omega(0)=0$. Using~\eqref{eq-preliminary-keylemma-1-1} and~\eqref{eq-preliminary-keylemma-1-2}, we have
    \begin{multline*}
    (s-1)c+b+\sum_{j=1}^s \big(d_j+\chi(1\leq \omega(j)\leq r)\chi(1\leq \omega(j-1)\leq r)\big)\\
    =\sum_{j=1}^s (k_{\omega(j)}-k_{\omega(j-1)})
    =k_{\omega(s)}
    \leq (s-1)c+b+e,
    \end{multline*}
    Therefore,
    \[
    \sum_{j=1}^s \big(d_j+\chi(1\leq \omega(j)\leq r)\chi(1\leq \omega(j-1)\leq r)\big)\leq e.
    \]
    This establishes the second inequality of~\eqref{eq-preliminary-keylemma-1-t}.
    
    To obtain the other inequality, let $1\leq i_1<\ldots<i_r\leq s$ be integers such that $\{\omega(i_k):k=1,\ldots,r\}=\{1,\ldots,r\}$, and define $i_0:=0$.
    To show 
    $r\leq \sum_{j=1}^{s}\big(d_j+\chi(1\leq \omega(j)\leq r)\chi(1\leq \omega(j-1)\leq r)\big)$, it suffices to show that 
    \[
    \sum_{k=0}^{r-1} \sum_{j=i_k+1}^{i_{k+1}}\left(\chi(\omega(j-1)<\omega(j))+\chi(1\leq \omega(j)\leq r)\chi(1\leq \omega(j-1)\leq r)\right)\geq r.
    \]
    This can be obtained by showing that for every $k=0,\ldots,r-1$,
    \begin{equation}\label{eq-preliminaries-keylemma-222} \sum_{j=i_k+1}^{i_{k+1}}\left(\chi(\omega(j-1)<\omega(j))+\chi(1\leq \omega(j)\leq r)\chi(1\leq \omega(j-1)\leq r)\right)\geq 1.
    \end{equation}
    The case $k=0$ is straightforward. In this case $i_k+1=i_0+1=1$. For $j=1$, the first term in the above summand, we have
    \[
    \chi(\omega(j-1)<\omega(j))=\chi(\omega(0)<\omega(1))=1.
    \]
    Hence, ~\eqref{eq-preliminaries-keylemma-222} holds. In the case when $1 \leq k \leq r-1$, if $i_{k+1}=i_k+1$, for $j=i_{k+1}$, we have \[
    \chi(1\leq \omega(j)\leq r)\chi(1\leq \omega(j-1)\leq r)=\chi(1\leq \omega(i_{k+1})\leq r)\chi(1\leq \omega(i_{k})\leq r)=1.
    \]
    If $i_k<i_k+1<i_{k+1}$, then $\omega(i_k+1)>r\geq \omega(i_k)$ and for  $j=i_k+1$ we have
    \[
    \chi(\omega(j-1)<\omega(j))=\chi(\omega(i_k)<\omega(i_k+1))=1.
    \]
    Thus, ~\eqref{eq-preliminaries-keylemma-222} always holds.
    This completes the proof. 
\end{proof}

By taking $r=1$ in Lemma~\ref{lem-preliminary-key} we obtain the following corollary, which is essentially~\cite[Lemma 8.1]{Zhou-AFLT}.
\begin{cor}\label{cor-preliminary-key}
For positive integers $s$ and $c$, let $b,e$ and $k_1,\dots,k_s$ be nonnegative integers such that $1\leq k_{i}\leq (s-1)c+b+e$ for $i=1,\ldots,s$.
Then at least one of the following holds:
\begin{enumerate}
\item[(i)] $1\leq k_i\leq b$ for some $i$ with $1\leq i\leq s$;
\item[(ii)] $-c\leq k_i-k_j\leq c-1$ for some $(i,j)$ such that $1\leq i<j\leq s$;
\item[(iii)] there exists a permutation $\omega\in\mathfrak{S}_s$ and nonnegative integers $d_1,\dots,d_s$
such that
\begin{subequations}\label{eq-preliminary-keycor-1}
\begin{equation}\label{eq-preliminary-keycor-1-1}
k_{\omega(1)}=b+d_1,
\end{equation}
and
\begin{equation}\label{eq-preliminary-keycor-1-2}
k_{\omega(j)}-k_{\omega(j-1)}=c+d_j \quad \text{for $2\leq j\leq s$.}
\end{equation}
\end{subequations}
Here the $d_j$ satisfy
\begin{equation}\label{eq-preliminary-keycor-range-t}
1\leq \sum_{j=1}^{s}d_j\leq e,
\end{equation}
$\omega(0):=0$, and $d_j>0$ if $\omega(j-1)<\omega(j)$ for $1\leq j\leq s$.
In particular, if $e=1$ then
\begin{equation}\label{eq-preliminary-keycor-t-specialcase}
k_i=(s-i)c+b+1
\end{equation}
for $i=1,\dots,s$.
\end{enumerate}
\end{cor}
Note that when $e=0$, only the first two cases can occur in Corollary~\ref{cor-preliminary-key}.

By taking certain $e$ in Lemma~\ref{lem-preliminary-key}, only the first three cases in the lemma can occur. This is the next corollary.
\begin{cor}\label{cor-preliminary-key-version2}
For positive integers $s$ and $c$, let $n,m,b$ and $k_1,\dots,k_s$ be nonnegative integers such that $s,m\leq n$,
and $1\leq k_{i}\leq (s-1)c+b+\chi(s>n-m+1)(s-n+m-1)$ for $i=1,\ldots,s$. Then for a nonnegative integer $r$ such that $s-n+m\leq r\leq \min\{s,m\}$, at least one of the following holds:
\begin{enumerate}
\item[(i)] $1\leq k_i\leq b$ for some $i$ with $1\leq i\leq s$;
\item[(ii)] $-c-1\leq k_i-k_j\leq c$ for some $(i,j)$ such that $1\leq i<j\leq r$;
\item[(iii)] $-c\leq k_{i}-k_{j}\leq c-1$ for some $(i,j)$ such that $1\leq i<j\leq s$ and $j>r$.
\end{enumerate}
\end{cor}
\begin{proof}
    It is routine to check by categorical discussion that if the case (iv) of Lemma~\ref{lem-preliminary-key} occurs, then there is a contradiction to~\eqref{eq-preliminary-keylemma-1-t}.
\end{proof}

\subsection{The cardinality of an alphabet}
To simplify the notation, for a positive integer $s$, let
\begin{equation}
T_{s}:=(s-1)c+b+\chi(s>n-m+1)(s-n+m-1).
\end{equation}
Note that we suppress the parameters $n,m,b,c$ in $T_{s}$.
\begin{lem}\label{lem-preliminary-cardi-P_mu}
For positive integers $s$, $n$, $m$ such that $m\leq n$ and $n-m+1\leq s\leq n$, let $U=\{u_1,\ldots,u_s\}$ be a set of integers with $1\leq u_1<\cdots<u_s\leq n$. Let $c$ be a positive integer, $b,\mathfrak{e}$ and $k_1,\ldots,k_s$ be nonnegative integers, such that $1\leq k_i\leq T_{s}+\mathfrak{e}$ for $i=1,\ldots,s$. If the $k_i$ are such that the case (iv) of Lemma~\ref{lem-preliminary-key} holds with $e=\chi(s>n-m+1)(s-n+m-1)+\mathfrak{e}$ and $r=|\{u_i:u_i\leq m\}|$, then
\begin{equation}\label{eq-preliminary-cardi}
    \bigg(\frac{q^{c-b}-q^a}{1-q^{c+1}}x_0+\sum_{i=1}^{n}\frac{q^{\chi(i>m)}-q^{c+1}}{1-q^{c+1}}x_i
    \bigg) \bigg|_{\substack{-a=T_{s}+\mathfrak{e}, \\[1pt] x_{u_i}=q^{k_s-k_i}x_{u_s},\,0\leq i\leq s}}
    \end{equation}
    is of the form
    \begin{equation}
        -\cfrac{1-q}{1-q^{c+1}}\big(\alpha_1+\cdots+\alpha_{\mathfrak{e}-1}\big)+\big(\beta_1+\cdots+\beta_{n-s}).
    \end{equation}
    Here each $\alpha_k$ is of the form $x_i q^j$, and $\beta_k$ is a single-letter alphabet, and $u_0=k_0:=0$.
\end{lem}
\begin{proof}
    By the definition of $r$, we have $u_i>m$ if and only if $i>r$. Hence,
    \begin{equation}\label{eq-preliminary-cardi-eq1}
    \begin{aligned}
        & \left(\frac{q^{c-b}-q^a}{1-q^{c+1}}x_0
        +\sum_{i=1}^{s}\frac{q^{\chi(u_i>m)}-q^{c+1}}{1-q^{c+1}}x_{u_i}\right)
        \bigg|_{\substack{-a=T_{s}+\mathfrak{e}, \\[1pt] x_{u_i}=q^{k_s-k_i}x_{u_s},\,0\leq i\leq s}}\\
        & = \cfrac{1-q}{1-q^{c+1}}x_{u_s}q^{k_s}
        \Big(  
        \cfrac{q^{c-b}-q^{-T_{s}-\mathfrak{e}}}{1-q}+\sum_{i=1}^s \cfrac{q^{\chi(i>r)}-q^{c+1}}{1-q}q^{-k_i}
        \Big)\\
        & = \cfrac{1-q}{1-q^{c+1}}x_{u_s}q^{k_s} \Big( -\sum_{j=-T_{s}-\mathfrak{e}}^{c-b-1} q^j+ \sum_{i=1}^s \ \sum_{j=-k_i+\chi(i>r)}^{-k_i+c} q^j \Big)\\
        & = \cfrac{1-q}{1-q^{c+1}}x_{u_s}q^{k_s} \Big( -\sum_{j=-T_{s}-\mathfrak{e}}^{c-b-1} q^j+ \sum_{i=1}^s \ \sum_{j=-k_{\omega(i)}+\chi(\omega(i)>r)}^{-k_{\omega(i)}+c} q^j \Big),
    \end{aligned}
    \end{equation}
    where $\omega\in\mathfrak{S}_s$ is the unique permutation such that $k_{\omega(1)}<\cdots<k_{\omega(s)}$. It follows from~\eqref{eq-preliminary-keylemma-1-2} that
    \[
    k_{\omega(i)}-k_{\omega(i-1)}\geq c+\chi(\omega(i)\leq r)\chi(\omega(i-1)\leq r)+\chi(\omega(i-1)<\omega(i))
    \]
    for $i=2,\ldots,s$. As it is routine to check that
    \[
    \chi(\omega(i-1)>r)+\chi(\omega(i)\leq r)\chi(\omega(i-1)\leq r)+\chi(\omega(i-1)<\omega(i))>0,
    \]
    therefore
    \[
    -k_{\omega(i-1)}+\chi(\omega(i-1)>r)>-k_{\omega(i)}+c
    \]
    for $i=2,\ldots,s$.
    Hence, the sum $\sum_{i=1}^s \ \sum_{j=-k_{\omega(i)}+\chi(\omega(i)>r)}^{-k_{\omega(i)}+c} q^j $ has $sc+r$ distinct summands. The maximal power of $q$ in this sum is $-k_{\omega(1)}+c$, which is no more than $c-b-1$ by~\eqref{eq-preliminary-keylemma-1-1}. The minimal power of $q$ in the sum is $-k_{\omega(s)}+\chi(\omega(s)>r)$, which is no less than $-T_{s}-\mathfrak{e}$. So all the summands in $\sum_{i=1}^s \ \sum_{j=-k_{\omega(i)}+\chi(\omega(i)>r)}^{-k_{\omega(i)}+c} q^j $ are eliminated by the summands in $-\sum_{j=-T_{s}-\mathfrak{e}}^{c-b-1} q^j$. Therefore, we can write
    \begin{multline}\label{eq-preliminary-cardi-eq2}
        \cfrac{1-q}{1-q^{c+1}}x_{u_s}q^{k_s} \Big( -\sum_{j=-T_{s}-\mathfrak{e}}^{c-b-1} q^j+ \sum_{i=1}^s \ \sum_{j=-k_{\omega(i)}+\chi(\omega(i)>r)}^{-k_{\omega(i)}+c} q^j \Big)=\\
        -\cfrac{1-q}{1-q^{c+1}}x_{u_s}q^{k_s}\Big(q^{j_1}+\cdots+q^{j_l}\Big),
    \end{multline}
    where $l=T_{s}+\mathfrak{e}-(s-1)c-b-r=(s-n+m-1)+\mathfrak{e}-r$, and $j_1,\ldots,j_{l}\in\{-T_{s}-\mathfrak{e},\ldots,c-b-1\}$.

    On the other hand, it is easy to see that
    \begin{equation}\label{eq-preliminary-cardi-eq3}
    \sum_{\substack{i=1\\i\notin U}}^n \cfrac{q^{\chi(i>m)}-q^{c+1}}{1-q^{c+1}}x_i
    =\sum_{\substack{i=1\\ i\notin U}}^n x_i -\cfrac{1-q}{1-q^{c+1}} \sum_{\substack{i=m+1\\ i\notin U}}^n x_i.
    \end{equation}
    Combining~\eqref{eq-preliminary-cardi-eq1},~\eqref{eq-preliminary-cardi-eq2}, ~\eqref{eq-preliminary-cardi-eq3}, we know that in~\eqref{eq-preliminary-cardi}, there are $l+(n-s-m+r)=\mathfrak{e}-1$ summands with the same multiple $-(1-q)/(1-q^{c+1})$, and the result follows.
\end{proof}

%% file: 5.proof1.tex
\section{Proof of theorem~\ref{Proposition-main}}\label{sec-proof1}
In this section we prove Theorem~\ref{Proposition-main}. The polynomiality of $M_n(a,b,c,\nu,\alpha)$ trivially follows from Lemma~\ref{lem-preliminary-ct-polynomiality} by taking $\tau=|\nu|-|\alpha|$ and
\[
L(x_1,\ldots,x_n)=x^\alpha E_\nu(x^{-1};q^{-1},q^{-c})\prod_{1\leq i<j\leq n}(x_i/x_j)_c(qx_j/x_i)_c
\]
in the lemma. It remains to establish the vanishing property of $M_n(a,b,c,\nu,\alpha)$. For nonnegative integers $d,b,c$, compositions $\alpha=(\alpha_1,\ldots,\alpha_n)$ and $\nu=(\nu_1,\ldots,\nu_n)$, define
\begin{multline}\label{definition-mainprop-Q(d)} 
Q(d)=Q_n(d,b,c,\nu,\alpha)\\
:=x_0^{|\nu|-|\alpha|} x_1^{\alpha_1}\cdots x_n^{\alpha_n} E_\nu(x^{-1};q^{-1},q^{-c})\prod_{i=1}^{n}\frac{(qx_{i}/x_{0})_b}{(q^{-d}x_{0}/x_{i})_d}
\prod_{1\leq i<j\leq n}
(x_i/x_j)_c(qx_j/x_i)_c.
\end{multline}
It is clear that
\[
\CT_x Q(-a)=M_n(a,b,c,\nu,\alpha)
\]
holds for all integers $a$. Therefore, to complete the proof of Theorem~\ref{Proposition-main}, it suffices to show that when $c\geq b\geq \nu_1^+$, the constant term $\CT\limits_x Q(d)$ vanishes for $d\in B_1(\nu^+)$. 

\subsection{The rational functions $Q(d)$ and $Q(d\Mid u^{(s)};k^{(s)})$}\label{subsec-proof1-1} We first give some essential properties of $Q(d)$. Observe that the numerator of $Q(d)$ ---
\[
x_0^{|\nu|-|\alpha|} x_1^{\alpha_1}\cdots x_n^{\alpha_n} E_\nu(x^{-1};q^{-1},q^{-c})\prod_{i=1}^{n}(qx_{i}/x_{0})_b
\prod_{1\leq i<j\leq n}
(x_i/x_j)_c(qx_j/x_i)_c
\]
is a Laurent polynomial in $x_0$ of degree at most $|\nu|$.
The denominator of $Q(d)$ ---
\[
\prod_{i=1}^{n}\big(q^{-d}x_{0}/x_{i}\big)_{d}
\]
is of the form
\[
\prod_{r=1}^{nd} (1-c_r x_0/ x_{i_r}),
\]
where all the $i_r\neq 0$, and $c_r\neq c_v$ whenever $i_r=i_v$. It is a polynomial in $x_0$ of degree $nd$. Thus, for a positive integer $d$, if $nd>|\nu|$, then $Q(d)$ is a rational function of the form \eqref{eq-preliminary-laruent-proper-1} with respect to $x_0$. We can therefore apply Lemma~\ref{lem-preliminary-laruent-proper} to $Q(d)$ and obtain a summation formula.

To express this sum, we first introduce some notation. For a positive integer $s$ such that $1\leq s\leq n$, let
$k=k^{(s)}:=(k_1, k_2, \dots, k_s)$ and $u=u^{(s)}:=(u_1,
u_2, \dots, u_s)$ be two positive integer sequences. For any rational function $F$ of $x_0, x_1, \dots, x_n$, denote the substitution
\begin{equation}\label{def-E_uk}
S_{u,k}F := F\mid_{x_{u_i}\mapsto x_{u_s}q^{k_s-k_i}
\text{, for $i = 0,
1,\dots, s-1$}},
\end{equation}
where we set $u_0 = k_0 = 0$. For $1 \leq u_1 <
u_2 <\dots < u_s \leq n$ and $1\leq k_i\leq d$, define
\begin{equation}\label{definition-mainprop-Q(d|u;k)}
Q(d\Mid u^{(s)};k^{(s)})=Q(d\Mid u_1,\dots,u_s;k_1,\dots,k_s):=S_{u,k}
\bigg(Q(d)\prod_{i=1}^{s}(1-\frac{x_{0}}{x_{u_{i}}q^{k_{i}}})\bigg).
\end{equation}
Note that the product on the right hand side of \eqref{definition-mainprop-Q(d|u;k)}
cancels all the factors in the denominator of $Q$ that would vanish under $S_{u,k}$.

Now if $nd>|\nu|$, we can apply Lemma~\ref{lem-preliminary-laruent-proper}
to $Q(d)$ with respect to $x_0$, and obtain
\begin{equation}\label{eq-mainprop-Q(d)-1}
\CT_{x_0}Q(d)=\sum_{\substack{1\leq k_1\leq d\\1\leq u_1\leq n}}
Q(d\Mid u_1;k_1),
\end{equation}
where $Q(d\Mid u_1;k_1)$ is defined in \eqref{definition-mainprop-Q(d|u;k)} with $s=1$. We can further apply Lemma~\ref{lem-preliminary-laruent-proper} to each $Q(d\Mid u_1;k_1)$ with respect to $x_{u_1}$ if applicable, and get a sum. Continue this process until Lemma~\ref{lem-preliminary-laruent-proper} does not apply to every summand. In other words, every summand cannot be written as a sum by Lemma~\ref{lem-preliminary-laruent-proper}. Finally we arrive at
\begin{equation}\label{eq-mainprop-Q(d)-sum}
\CT_{x}Q(d)=\sum_{s\in T\subseteq \{1,\dots,n\}}\sum_{\substack{1\leq u_1<\cdots<u_s\leq n\\1\leq k_1,\dots,k_s\leq d}}
\CT_xQ(d\Mid u_1,\dots,u_s;k_1,\dots,k_s).
\end{equation}
This process, first introduced by Gessel and Xin \cite{GX}, is usually referred to as the Gessel--Xin operation to the rational function $Q(d)$.

Let $U=\{u_1,\ldots,u_s\}$. We write the explicit expression of $Q(d\Mid u^{(s)};k^{(s)})$ as
\begin{equation}\label{eq-mainprop-Q(d|r;k)-HVLform}
Q(d\Mid u^{(s)};k^{(s)})=N\times V \times L\\
\times\prod_{\substack{1\leq i<j\leq n\\i,j\notin U}}\big(x_i/x_j\big)_{c}
\big(qx_j/x_i\big)_{c},
\end{equation}
where
\[
N=S_{u,k}\bigg(x_0^{|\nu|-|\alpha|} x_1^{\alpha_1}\cdots x_n^{\alpha_n} E_\nu(x^{-1};q^{-1},q^{-c})\bigg),
\]
\[
V=\prod_{i=1}^s(q^{k_i-d})^{-1}_{d-k_i}(q)^{-1}_{k_i-1}(q^{1-k_i})_{b}\prod_{1\leq i<j\leq s}(q^{k_j-k_i})_c(q^{k_i-k_j+1})_c,
\]
and
\begin{equation}\label{eq-mainprop-Q(d|r;k)-L}
L=\prod_{\substack{i=1\\ i\notin U}}^n
\frac{(q^{1-k_s}x_i/x_{u_s})_{b}}{\big(q^{k_s-d}x_{u_s}/x_i\big)_d}
\prod_{\substack{i=1\\ i\notin U}}^n\prod_{j=1}^s
\big(q^{k_j-k_s+\chi(i>u_j)}x_i/x_{u_s}\big)_c
\big(q^{k_s-k_j+\chi(u_j>i)}x_{u_s}/x_i\big)_c.
\end{equation}
Note that $V$ is a constant factor of $Q(d\Mid u^{(s)};k^{(s)})$, and the variable $x_{u_s}$ appears only in $N$ and $L$.

The next proposition shows that, under certain conditions, the rational function $Q(d\Mid u^{(s)};k^{(s)})$ can be written as a Laurent polynomial in $x_{u_s}.$ The proof of this proposition is straightforward: it follows by verifying that the numerator of $L$ cancels with its denominator. For a detailed proof, we refer the reader to~\cite[Proposition 8.4]{Zhou-AFLT}.

\begin{prop}\label{prop-mainprop-Q(d|r;k)-polynomial}
Let $1\leq d\leq sc+b$. Then for $i=1,\ldots,s$, the $k_i$ in $Q(d\Mid u^{(s)};k^{(s)})$ satisfy $1\leq k_i\leq d\leq sc+b$. By Corollary~\ref{cor-preliminary-key} with $e=c$, only three cases can occur for the \( k_i \). If the third case holds, then the rational function $L$ in \eqref{eq-mainprop-Q(d|r;k)-L} can be written as a Laurent polynomial in $x_{u_s}$ of the form
\begin{equation}\label{eq-mainprop-Q(d|r;k)-polynomialform}
x_{u_s}^{(n-s)(sc-d)}\prod_{\substack{i=1\\ i\notin U}}^n
\Big(p_i(x_i/x_{u_s})x_i^{d-sc}\Big),
\end{equation}
where the $p_i(z)$ are polynomials in $z$.
\end{prop}

The next proposition gives vanishing and recursive properties for $Q(d\Mid u^{(s)};k^{(s)})$. 
\begin{prop}\label{prop-mainprop-Q(d|u;k)-properties}
Let $s$ be a positive integer such that $1\leq s\leq n$.
\begin{enumerate}
\item If $d\leq (s-1)c+b$, then $Q(d\Mid u^{(s)};k^{(s)})=0$;

\item If $d>sc+\sum_{k=1}^{n-s}\nu^+_k/(n-s)$ and $s\neq n$, then
\begin{equation}\label{prop-mainprop-Q(d|u;k)-properties-propercase}
\CT_{x_{u_s}}Q(d\Mid u^{(s)};k^{(s)})=
\begin{cases}\displaystyle
\sum_{\substack{u_s<u_{s+1}\leq n\\1\leq k_{s+1}\leq d}}
Q(d\Mid u_1,\dots,u_s,u_{s+1};k_1,\dots,k_s,k_{s+1}) \quad &\text{for $u_s<n$,}\\
0 \quad &\text{for $u_s=n$}.
\end{cases}
\end{equation}
\end{enumerate}
\end{prop}

\begin{proof}
(1) Since $d\leq (s-1)c+b$, for $i=1,\dots,s$ we have
\[
1\leq k_i\leq d \leq (s-1)c+b.
\]
By Corollary~\ref{cor-preliminary-key} with $e=0$, either $1\leq k_i\leq b$ for some $i$ with $1\leq i\leq s$,
or $-c\leq k_i-k_j\leq c-1$ for some $(i,j)$ such that $1\leq i<j\leq s$. If $1\leq k_i\leq b$ for some $i$, then $Q(d\Mid u^{(s)};k^{(s)})$ has the factor
\[
S_{u,k}\Big[\big(qx_{u_i}/x_0\big)_{b}\Big]=\big(q^{1-k_i}\big)_{b}=0.
\]
If $-c\leq k_i-k_j\leq c-1$ for some $(i,j)$, then $Q(d\Mid u^{(s)};k^{(s)})$ has the factor
\[
S_{u,k}\Big[\big(x_{u_i}/x_{u_j}\big)_c\big(qx_{u_j}/x_{u_i}\big)_c\Big],
\]
which is equal to
\[
S_{u,k}\Big[q^{\binom{c+1}{2}}(-x_{u_j}/x_{u_i})^c\big(q^{-c}x_{u_i}/x_{u_j}\big)_{2c}\Big]
=q^{\binom{c+1}{2}}(-q^{k_i-k_j})^c\big(q^{k_j-k_i-c}\big)_{2c}=0.
\]
Therefore, $Q(d\Mid u^{(s)};k^{(s)})=0$ as it always contains a zero factor.

(2) We first show that $Q(d\Mid u^{(s)};k^{(s)})$ is of the form \eqref{eq-preliminary-laruent-proper-1} if $d>sc+\sum_{k=1}^{n-s}\nu^+_k/(n-s)$. Let $U=\{u_1,u_2,\dots,u_s\}$.  By the expression of $Q(d\Mid u^{(s)};k^{(s)})$ in~\eqref{eq-mainprop-Q(d|r;k)-HVLform}, the part that contributes to the degree in $x_{u_s}$
in the numerator of $Q(d\Mid u^{(s)};k^{(s)})$ is
\begin{equation*}
S_{u,k}\bigg(x_0^{|\nu|-|\alpha| } x_1^{\alpha_1}\cdots x_n^{\alpha_n}  E_\nu(x^{-1};q^{-1},q^{-c})\bigg) \prod_{\substack{i=1\\ i\notin U}}^n\prod_{j=1}^s
\big(q^{k_s-k_j+\chi(u_j>i)}x_{u_s}/x_i\big)_c,
\end{equation*}
which has degree in $x_{u_s}$ at most $(n-s)sc+\sum_{k=1}^{n-s}\nu^+_k$ by Corollary~\ref{cor-nonmac-degree}. The part that contributes to the degree in $x_{u_s}$ in the denominator of $Q(d\Mid u^{(s)};k^{(s)})$ is
\[
\prod_{\substack{i=1\\ i\notin U}}^n
\big(q^{k_s-d}x_{u_s}/x_i\big)_d,
\]
which has degree $(n-s)d$. Therefore, when $d>sc+\sum_{k=1}^{n-s}\nu^+_k/(n-s)$, the rational function $Q(d\Mid u^{(s)};k^{(s)})$ is of the form \eqref{eq-preliminary-laruent-proper-1}.
Applying Lemma~\ref{lem-preliminary-laruent-proper} yields
\[
\CT_{x_{u_s}}Q(d\Mid u^{(s)};k^{(s)})=
\begin{cases}\displaystyle
\sum_{\substack{u_s<u_{s+1}\leq n\\1\leq k_{s+1}\leq d}}
Q(d\Mid u_1,\dots,u_s,u_{s+1};k_1,\dots,k_s,k_{s+1}) \quad &\text{for $u_s<n$,}\\
0 \quad &\text{for $u_s=n$.}
\end{cases}
\]
This completes the proof.
\end{proof}

\begin{cor}\label{cor-mainprop-Q(d)-i}
    Assume $c\geq b\geq \nu_1^+$. If there is an $i\in\{1,\ldots,n-1\}$ such that $ic+\nu^+_{n-i}+1\leq d\leq ic+b$, then
    \begin{equation}\label{eq-mainprop-Q(d)-i}
    \CT_{x}Q(d)=\sum_{\substack{1\leq u_1<\cdots<u_{i}\leq n\\1\leq k_1,\dots,k_{i}\leq d}}
    \CT_x Q(d\Mid u^{(i)};k^{(i)}).
    \end{equation}
\end{cor}
\begin{proof}
Under the given conditions, we have 
$nd> nc\geq |\nu|$, and thus the expansion~\eqref{eq-mainprop-Q(d)-sum} holds. To complete the proof, it suffices to show that the index $s$ in each nonzero term $Q(d\Mid u^{(s)};k^{(s)})$ in~\eqref{eq-mainprop-Q(d)-sum} can only be $i$.

If $s\geq i+1$, we have $d\leq ic+b \leq (s-1)c+b$. By the case (1) of Proposition~\ref{prop-mainprop-Q(d|u;k)-properties} we have
$Q(d\Mid u^{(s)};k^{(s)})=0$.
If $i-1 \geq s$, then together with $c\geq \nu_1^+$, we have
\[
d\geq ic+\nu^+_{n-i}+1\geq (s+1)c+1 > sc+\sum_{k=1}^{n-s}\nu^+_k/(n-s).
\]
It is clear that $s\neq n$ since $s\leq i-1\leq n-2$.
By the case (2) of Proposition~\ref{prop-mainprop-Q(d|u;k)-properties}, the summand $\CT\limits_x Q(d\Mid u^{(s)};k^{(s)})$ is either zero or can be written as a sum (against that $s$ is maximal). Hence, the index $s$ in each nonzero term $Q(d\Mid u^{(s)};k^{(s)})$ in~\eqref{eq-mainprop-Q(d)-sum} can only be $i$. 
\end{proof}

\subsection{Proof of Theorem~\ref{Proposition-main}}
We now complete the proof of Theorem~\ref{Proposition-main} by showing that when $c\geq b\geq \nu_1^+$, 
\[
\CT\limits_x Q(d)=0\quad  \text{ for $d\in B_1(\nu^+)$ }.
\]
Here we recall that
\[
B_1(\nu^+)=\bigcup_{i=0}^{n-1} 
\{ic+\nu_{n-i}^++1,\ldots,ic+b\}.
\]
\begin{proof}[Proof of Theorem~\ref{Proposition-main}]
We first prove that $\CT\limits_x Q(d)=0$ for $\nu^+_n+1\leq d\leq b$. In this case, the rational function $Q(d)$ can be rewritten as
\begin{multline}\label{eq-mainprop-Q(d)-d<b-case}
Q(d)=x_0^{|\nu|-|\alpha| } x_1^{\alpha_1}\cdots x_n^{\alpha_n}  E_\nu(x^{-1};q^{-1},q^{-c})\\
\times \prod_{i=1}^{n}(-x_i/x_0)^d q^{\binom{d+1}{2}} (q^{d+1}x_i/x_0)_{b-d}
\prod_{1\leq i<j\leq n}
(x_i/x_j)_c(qx_j/x_i)_c.
\end{multline}
By taking the constant term with respect to $x_0$ in (\ref{eq-mainprop-Q(d)-d<b-case}), we express $\CT\limits_{x} Q(d)$ as a finite sum of terms of the form
\[
\CT_x x_1^{v_1}\cdots x_n^{v_n} E_\nu(x^{-1};q^{-1},q^{-c})\prod_{1\leq i<j\leq n}
(x_i/x_j)_c(qx_j/x_i)_c,
\]
where $v=(v_1,\ldots,v_n)\in\mathbb{N}^n$ satisfies $|v|=|\nu|$, and $v_i\geq d\geq \nu^+_n+1$ for $i=1,\ldots,n$. 
Observe that $\nu^+\not\leq v^+$, and so $\nu\not\preceq v$. By Lemma~\ref{lem-nonmac-vanish-property} (with $t=q^c$), we have
\[
\CT_x  x_1^{v_1}\cdots x_n^{v_n} E_\nu(x^{-1};q^{-1},q^{-c})\prod_{1\leq i<j\leq n}
(x_i/x_j)_c(qx_j/x_i)_c=\langle x^v, E_\nu\rangle_{q,q^c,n}=0.
\]
Therefore, $\CT\limits_x Q(d)=0$.

Now fix $i\in\{1,\ldots,n-1\}$. We show that the constant term $\CT\limits_x Q(d)=0$ for $ic+\nu^+_{n-i}+1\leq d\leq ic+b$. In this case, 
\begin{equation}\label{eq-mainprop-q(d)-i}
\CT_{x}Q(d)=\sum_{\substack{1\leq u_1<\cdots<u_{i}\leq n\\1\leq k_1,\dots,k_{i}\leq d}} \CT_x Q(d\Mid u^{(i)};k^{(i)})
\end{equation}
by Corollary~\ref{cor-mainprop-Q(d)-i}. Thus, to show $\CT\limits_x Q(d)=0$, it suffices to show $\CT\limits_x Q(d\Mid u^{(i)};k^{(i)})=0$ for every summand in the above sum. Since $ic+\nu^+_{n-i}+1\leq d\leq ic+b$ and $b\leq c$, we can write $d=(i-1)c+b+e$ for a positive integer $e$.
It follows that the $k_j$ in $Q(d\Mid u^{(i)};k^{(i)})$ satisfy 
\[
1\leq k_j\leq d=(i-1)c+b+e\leq ic+b 
\]
for $j=1,\dots,i$. By Corollary~\ref{cor-preliminary-key} with $s=i$, at least one of the following holds:
\begin{enumerate}
\item[(i)] $1\leq k_j\leq b$ for some $j$ with $1\leq j\leq i$;
\item[(ii)] $-c\leq k_r-k_l\leq c-1$ for some $(r,l)$ such that $1\leq r<l\leq i$;
\item[(iii)] there exists a permutation $w\in\mathfrak{S}_{i}$ and nonnegative integers $d_1,\dots,d_{i}$
such that
\[
k_{w(1)}=b+d_1,
\]
and
\[
k_{w(j)}-k_{w(j-1)}=c+d_j \quad \text{for $2\leq j\leq i$.}
\]
Here the $d_j$ satisfy
\[
\sum_{j=1}^{i}d_j\leq e,
\]
$w(0):=0$, and $d_j>0$ if $w(j-1)<w(j)$ for $1\leq j\leq i$.
\end{enumerate}
If either of the first two cases holds, by the same argument as that in the proof of the case (1) of Proposition~\ref{prop-mainprop-Q(d|u;k)-properties}, we have $\CT\limits_x Q(d\Mid u^{(i)};k^{(i)})=0$. If the case (iii) holds, let $U:=\{u_1,\dots,u_{i}\}$, $U^c:=\{1,\ldots,n\}\setminus U=\{j_1,\ldots,j_{n-i}\}$ with the $j_l$ in an increasing order. By Proposition~\ref{prop-mainprop-Q(d|r;k)-polynomial} with $s=i$, we can rewrite $Q(d\Mid u^{(i)};k^{(i)})$ as a Laurent polynomial in $x_{u_{i}}$. That is, 
\begin{equation}\label{eq-mainprop-Q(d|r;k)-s=i-polynomial}
\begin{aligned}
\CT\limits_x Q(d\Mid u^{(i)};k^{(i)})&= q^{k_i(|\nu|-|\alpha|)}\prod_{l=1}^i(q^{k_l-d})^{-1}_{d-k_l}(q)^{-1}_{k_l-1}(q^{1-k_l})_{b}\prod_{1\leq r<l\leq i}(q^{k_l-k_r})_c(q^{k_r-k_l+1})_c\\
&\quad\times \CT\limits_x x_{u_{i}}^{|\nu|-|\alpha|+(n-i)(ic-d)} S_{u^{(i)},k^{(i)}}\big(x_1^{\alpha_1}\cdots x_n^{\alpha_n} E_\nu(x^{-1};q^{-1},q^{-c})\big)
\\
&\quad\times \prod_{l=1}^{n-i}
\big(p_l(x_{j_l}/x_{u_{i}})x_{j_l}^{d-ic}\big)
\prod_{1\leq r<l\leq n-i}(x_{j_r}/x_{j_l})_c
(qx_{j_l}/x_{j_r})_c,
\end{aligned}
\end{equation}
where the $p_l(z)$ are polynomials in $z$. By Lemma~\ref{cor-nonmac-splitvariables} with $I\mapsto U$, $J\mapsto U^c$, $x\mapsto x^{-1}$, $t\mapsto q^c$, we have
\begin{multline*}
E_\nu(x_1^{-1},\ldots,x_n^{-1};q^{-1},q^{-c})=\\
\sum_{\substack{\rho\in\mathbb{N}^s,\gamma\in\mathbb{N}^{n-s}\\\rho^+,\gamma^+\subseteq \nu^+}} c_{\rho\gamma}^\nu(U;U^c) E_{\rho}(x_{u_1}^{-1},\ldots,x_{u_i}^{-1};q^{-1},q^{-c}) E_{\gamma}(x_{j_1}^{-1},\ldots,x_{j_{n-i}}^{-1};q^{-1},q^{-c}),
\end{multline*}
where $c_{\rho\gamma}^\nu(U;U^c)\in \mathbb{F}$. Hence,
\begin{multline*}
S_{u^{(i)},k^{(i)}}\big(E_\nu(x_1^{-1},\ldots,x_n^{-1};q^{-1},q^{-c})\big)=\\
\sum_{\substack{\rho\in\mathbb{N}^s,\gamma\in\mathbb{N}^{n-s}\\\rho^+,\gamma^+\subseteq \nu^+}} c_{\rho\gamma}^\nu(U;U^c) x_{u_i}^{-|\rho|}E_{\rho}(q^{k_1-k_i},\ldots,q^{k_{i-1}-k_i},1;q^{-1},q^{-c}) E_\gamma(x_{j_1}^{-1},\ldots,x_{j_{n-i}}^{-1};q^{-1},q^{-c}).
\end{multline*}
By taking the constant term with respect to $x_{u_i}$ in \eqref{eq-mainprop-Q(d|r;k)-s=i-polynomial}, we write $\CT\limits_x Q(d\Mid u^{(i)};k^{(i)})$ as a finite sum, up to a constant multiple, each summand is of the form
\[
\CT\limits_x x_{j_1}^{v_1}\cdots x_{j_{n-i}}^{v_{n-i}}  E_\gamma(x_{j_1}^{-1},\ldots,x_{j_{n-i}}^{-1};q^{-1},q^{-c}) \prod_{1\leq r<l\leq n-i}\big(x_{j_r}/x_{j_l}\big)_c
\big(qx_{j_l}/x_{j_r}\big)_c,
\]
where $v,\gamma$ are compositions such that $|v|=|\gamma|$,  $v_k\geq d-ic\geq \nu^+_{n-i}+1$ for $k=1,\ldots,n-i$, and $\gamma^+\subseteq \nu^+, \ell(\gamma^+)\leq n-i$. Observe that $\gamma\not\preceq v$. Applying Lemma~\ref{lem-nonmac-vanish-property} again, we obtain 
\begin{multline*}
    \CT_x x_{j_1}^{v_1} \cdots x_{j_{n-i}}^{v_{n-i}}  E_\gamma(x_{j_1}^{-1},\ldots,x_{j_{n-i}}^{-1};q^{-1},q^{-c}) \prod_{1\leq r<l\leq n-i}\big(x_{j_r}/x_{j_l}\big)_c
    \big(qx_{j_l}/x_{j_r}\big)_c\\
    =\langle x^v, E_\gamma\rangle_{q,q^c,n-i}=0.
\end{multline*}
Therefore $\CT\limits_x Q(d\Mid u^{(i)};k^{(i)})=0$. Now every summand $\CT\limits_x Q(d \Mid u^{(i)};k^{(i)})$ appearing on the right-hand side of ~\eqref{eq-mainprop-q(d)-i} is zero, and so is their finite sum $\CT\limits_x Q(d)$. This completes the proof.
\end{proof}

%% file: 6.proof2.tex
\section{Proof of Theorem~\ref{proposition-main-2-1} and Theorem~\ref{proposition-main-2-2}}\label{sec-proof2}
In this section, we prove Theorem~\ref{proposition-main-2-1} and Theorem~\ref{proposition-main-2-2}. In the first subsection, we establish Theorem~\ref{proposition-main-2-1}. In Subsection~\ref{subsec-proof2-2}, following an approach similar to that in the proof of Theorem~\ref{Proposition-main}, we introduce two rational functions associated with $F_{n,m}(a,b,c,H_\tau,\mu)$, and establish their key properties. Finally, in Subsection~\ref{subsec-proof2-3} we prove Theorem~\ref{proposition-main-2-2}.

\subsection{Polynomiality and Rationality}\label{subsec-proof2-1}
We first establish part (a) of Theorem~\ref{proposition-main-2-1}, which is the next lemma.

\begin{lem}[Polynomiality in $q^a$]\label{lem-F(H_tau,mu)-polynomial}
     Let $F_{n,m}(a,b,c,H_\tau,\mu)$ be defined in~\eqref{def-intro-F(mu,H)}. If $H_\tau$ is independent of $a$ and $x_0$, then for fixed nonnegative integers $b$ and $c$, the constant term $F_{n,m}(a,b,c,H_\tau,\mu)$ is a polynomial in $q^a$ of degree at most $nb+|\mu|-\tau$. 
\end{lem}
Note that by degree considerations, the constant term $F_{n,m}(a,b,c,H_\tau,\mu)$ vanishes whenever $nb+|\mu|-\tau<0$.
\begin{proof}
    By expanding the Macdonald polynomial $P_\mu$ in $F_{n,m}(a,b,c,H_\tau,\mu)$
    into a sum of monomials, we have
    \begin{equation}\label{e-cor-M-polynomial-1}
        P_\mu\left(\left[\cfrac{q^{c-b}-q^a}{1-q^{c+1}}x_0+X_m +\cfrac{q-q^{c+1}}{1-q^{c+1}}X_{n\setminus m}\right];q,q^{c+1}\right)=\sum_\alpha c_\alpha x_0^{\alpha_0}x_1^{\alpha_1}\cdots x_n^{\alpha_n},
    \end{equation}
    where $\alpha=(\alpha_0,\ldots,\alpha_n)\in\mathbb{N}^{n+1}$ ranges over all nonnegative integer sequences such that $|\alpha|=|\mu|$, and the coefficient $c_{\alpha}$ is a polynomial in $q^{a}$ of degree at most $\alpha_0$. Then, the constant term $F_{n,m}(a,b,c,H_\tau,\mu)$ can be expressed as 
    \begin{equation}\label{e-cor-M-polynomial-2}
        \sum_{\alpha} c_{\alpha}\CT_{x}x_0^{\tau-|\mu|+\alpha_0}L_\alpha(x_1,\dots,x_n)\prod_{i=1}^n (x_0/x_i)_a(qx_i/x_0)_b,
    \end{equation}
     where
    \begin{equation}\label{e-cor-M-polynomial-3}    L_\alpha(x_1,\dots,x_n)=x_1^{\alpha_1}\cdots x_n^{\alpha_n}
    H_\tau \prod_{1\leq i<j\leq n}(x_i/x_j)_{c+\chi(j\leq m)}(qx_j/x_i)_{c+\chi(j\leq m)}.
    \end{equation}
     Since each $L_\alpha(x_1,\dots,x_n)$ is independent of $a$ and $x_0$, and $\alpha_0\geq 0$, $\alpha_0-|\mu|\leq 0$, by Lemma \ref{lem-preliminary-ct-polynomiality} each summand in \eqref{e-cor-M-polynomial-2} is a polynomial in $q^a$ of degree at most $nb+|\mu|-\tau$, and so is the sum. 
\end{proof}

In the following lemma, we prove part (b) of Theorem~\ref{proposition-main-2-1}
\begin{lem}[Rationality in $q^{b}$]
     Let $F_{n,m}(a,b,c,H_\tau,\mu)$ be defined in~\eqref{def-intro-F(mu,H)}. If $H_\tau$ is independent of $b$ and $x_0$, then for fixed nonnegative integers $a$ and $c$, the constant term $F_{n,m}(a,b,c,H_\tau,\mu)$ is a rational function in $q^b$. 
\end{lem}
\begin{proof}
    Observe that for fixed $a$ and $c$, the coefficient $c_{\alpha}$ in \eqref{e-cor-M-polynomial-2} is a polynomial in $q^{-b}$. Since $H_\tau$ is independent of $b$ and $x_0$,  $L_\alpha(x_1,\dots,x_n)$ in \eqref{e-cor-M-polynomial-3} is independent of $b$ and $x_0$. By Lemma \ref{lem-preliminary-ct-polynomiality} the constant term in each summand in \eqref{e-cor-M-polynomial-2} is a polynomial in $q^b$. Hence, each summand in \eqref{e-cor-M-polynomial-2} is a rational function in $q^{b}$ and so is the sum. This completes the proof.
\end{proof}
We prove part (c) of Theorem~\ref{proposition-main-2-1} in the next lemma.
\begin{lem}[Rationality in $q^c$]\label{cor-F(H_tau,mu)-rationality}
For fixed nonnegative integers $a$ and $b$, let $F_{n,m}(a,b,c,H_\tau,\mu)$ be defined in~\eqref{def-intro-F(mu,H)}. Then
\[
\frac{(q)_c^n}{(q)_{nc}}\cdot F_{n,m}(a,b,c,H_\tau,\mu)
\]
is a rational function in $q^c$ and $q$.
\end{lem}
\begin{proof}
Rewrite $F_{n,m}(a,b,c,H_\tau,\mu)$ as
\begin{equation}\label{e-cor-F_-rational-c-1}
    \CT_x x_0^{\tau-|\mu|} S\times \prod_{i=1}^n (x_0/x_i)_a (qx_i/x_0)_b \prod_{1\leq i<j\leq n}(x_i/x_j)_c(qx_j/x_i)_c,
\end{equation}
where $S$ is
\[
\prod_{1\leq i<j\leq m} (1-q^c x_i/x_j)(1-q^{c+1}x_j/x_i) H_\tau(x) P_\mu\left(\left[\cfrac{q^{c-b}-q^a}{1-q^{c+1}}x_0+X_m +\cfrac{q-q^{c+1}}{1-q^{c+1}}X_{n\setminus m}\right];q,q^{c+1}\right).
\]
Observe that $S$ is a Laurent polynomial in $x_0,\ldots,x_n$ with coefficients in $\mathbb{F}$.
By expanding $S$ into a finite sum of monomials, and taking the constant term with respect to $x_0$ in~\eqref{e-cor-F_-rational-c-1}, we write
\begin{equation}
    F_{n,m}(a,b,c,H_\tau,\mu)=\sum_{\alpha} c_\alpha \CT_x x_1^{\alpha_1}\cdots x_n^{\alpha_n}\prod_{1\leq i<j\leq n}(x_i/x_j)_c(qx_j/x_i)_c,
\end{equation}
where the sum is finite, $\alpha\in\mathbb{Z}^n$, $|\alpha|=0$, and $c_\alpha\in\mathbb{F}$. The result follows by applying Proposition ~\ref{prop-preliminary-stembridge-rationality} to each summand above.
\end{proof}

\subsection{The two rational functions}\label{subsec-proof2-2} 
For nonnegative integers $d,b,c,\tau$, a partition $\mu$ such that $\ell(\mu)<m$, and a homogeneous Laurent polynomial $H_\tau=H_\tau(x)$ of the form~\eqref{def-intro-H_tau}, define
\begin{multline}\label{def-P(d)}
P(d)=P_{n,m}(d,b,c,H_\tau,\mu):= x_0^{\tau-|\mu|} P_\mu\left(\left[\cfrac{q^{c-b}-q^{-d}}{1-q^{c+1}}x_0+X_m+\cfrac{q-q^{c+1}}{1-q^{c+1}}X_{n\setminus m}\right];q,q^{c+1}\right)\\
\times H_\tau \prod_{i=1}^n \cfrac{(qx_i/x_0)_b}{(q^{-d}x_0/x_i)_d}\prod_{1\leq i<j\leq n} \left(x_i/x_j\right)_{c+\chi(j\leq m)} \left(qx_j/x_i\right)_{c+\chi(j\leq m)}.
\end{multline}
It is clear that the identity
\begin{equation}
    \CT_x P(-a)= F_{n,m}(a,b,c,H_\tau,\mu)
\end{equation}
holds for all integers $a$. Therefore, to complete the proof of Theorem~\ref{proposition-main-2-2}, it suffices to prove that when $c\geq b+\mu_1$ and $b\geq \tau+m$, the constant term $\CT\limits_x P(d)$ vanishes for $d\in B_2(\delta^m)\cup B_3(\mu)$. Here we recall that the sets $B_2$ and $B_3$ are defined in~\eqref{def-sets-B2} and~\eqref{def-sets-B3}, respectively.

For a positive integer $s$ such that $1\leq s\leq n$, let $k=k^{(s)} := (k_1, k_2, \dots, k_s)$ and $u=u^{(s)} :=(u_1, u_2, \dots, u_s)$ be two sequences of positive integers such that $0 < u_1 <
u_2 <\dots < u_s \leq n$ and $1\leq k_i\leq d$. Define
\begin{equation}\label{definition-mainthm-P(d|u;k)}
P(d\Mid u^{(s)};k^{(s)})=P(d\Mid u_1,\dots,u_s;k_1,\dots,k_s):=S_{u,k}
\bigg(P(d)\prod_{i=1}^{s}(1-\frac{x_{0}}{x_{u_{i}}q^{k_{i}}})\bigg),
\end{equation}
where $S_{u,k}$ is the substitution defined in~\eqref{def-E_uk}. Following the same argument as for $Q(d)$, when $nd>\tau$, by the Gessel--Xin operation we have
\begin{equation}\label{eq-mainthm-P(d)-sum}
\CT_{x}P(d)=\sum_{s\in T\subseteq \{1,\dots,n\}}\sum_{\substack{1\leq u_1<\cdots<u_s\leq n\\1\leq k_1,\dots,k_s\leq d}}
\CT_x P(d\Mid u_1,\dots,u_s;k_1,\dots,k_s),
\end{equation}
where each index $s$ in $P(d\Mid u^{(s)};k^{(s)})$ in the sum is maximal, in the sense that each summand cannot be written as a sum by Lemma~\ref{lem-preliminary-laruent-proper}. For a given $P(d\Mid u^{(s)};k^{(s)})$, we define an integer $r$ associated with $u^{(s)}$ by
\begin{equation}\label{definition-mainthm-r}
    r=r(u^{(s)}):=\Mid \{u_i:u_i\leq m,i=1,\ldots,s\} \Mid,
\end{equation}
where $|A|$ is the cardinality of the set $A$.
It is clear that
\begin{equation}\label{eq-mainthm-range-r}
\max\{s-n+m,0\}\leq r\leq \min\{s,m\}.
\end{equation}
\begin{prop}\label{prop-mainthm-P(d|u;k)-properties}
For an integer $s$ such that $1\leq s\leq n$, let $P(d\Mid u^{(s)};k^{(s)})$ and $r$ be defined in~\eqref{definition-mainthm-P(d|u;k)}  and ~\eqref{definition-mainthm-r}, respectively. The rational function $P(d\Mid u^{(s)};k^{(s)})$ has the following vanishing and recursive properties:
\begin{enumerate}
\item If $d\leq (s-1)c+b+\chi(s>n-m+1)(s-n+m-1)$, then $P(d\Mid u^{(s)};k^{(s)})=0$;

\item If $d>sc+\big((m-r)\cdot r+\tau\big)/(n-s)$ and $s\neq n$, then
\begin{equation}\label{prop-mainthm-P(d|u;k)-properties-propercase}
\CT_{x_{u_s}}P(d\Mid u^{(s)};k^{(s)})=
\begin{cases}\displaystyle
\sum_{\substack{u_s<u_{s+1}\leq n\\1\leq k_{s+1}\leq d}}
P(d\Mid u_1,\dots,u_s,u_{s+1};k_1,\dots,k_s,k_{s+1}) \quad &\text{for $u_s<n$},\\
0 \quad &\text{for $u_s=n$}.
\end{cases}
\end{equation}
\end{enumerate}
\end{prop}

\begin{proof}
(1)
If $d\leq (s-1)c+b+\chi(s>n-m+1)(s-n+m-1)$, then for $i=1,\ldots,s$, we have
\[
1\leq k_i\leq d \leq (s-1)c+b+\chi(s>n-m+1)(s-n+m-1).
\]
By Corollary~\ref{cor-preliminary-key-version2} with $r$ taken as in~\eqref{definition-mainthm-r}, at least one of the cases (i)-(iii) of Corollary~\ref{cor-preliminary-key-version2} holds.

If (i) holds, then $P(d\Mid u^{(s)};k^{(s)})$ has the factor
\[
S_{u,k}\Big[\big(qx_{u_i}/x_0\big)_{b}\Big]=\big(q^{1-k_i}\big)_{b}=0.
\]

If (ii) holds, then $P(d\Mid u^{(s)};k^{(s)})$ has the factor
\[
S_{u,k}\Big[\big(x_{u_i}/x_{u_j}\big)_{c+1}\big(qx_{u_j}/x_{u_i}\big)_{c+1}\Big],
\]
which is equal to
\[
S_{u,k}\Big[q^{\binom{c+2}{2}}(-x_{u_j}/x_{u_i})^{c+1}\big(q^{-c-1}x_{u_i}/x_{u_j}\big)_{2c+2}\Big]
=q^{\binom{c+2}{2}}(-q^{k_i-k_j})^{c+1}\big(q^{k_j-k_i-c-1}\big)_{2c+2}=0.
\]

If (iii) holds, then $P(d\Mid u^{(s)};k^{(s)})$ has the factor
\[
S_{u,k}\Big[\big(x_{u_i}/x_{u_j}\big)_c\big(qx_{u_j}/x_{u_i}\big)_c\Big],
\]
which is equal to
\[
S_{u,k}\Big[q^{\binom{c+1}{2}}(-x_{u_j}/x_{u_i})^c\big(q^{-c}x_{u_i}/x_{u_j}\big)_{2c}\Big]
=q^{\binom{c+1}{2}}(-q^{k_i-k_j})^c\big(q^{k_j-k_i-c}\big)_{2c}=0.
\]
Therefore, we can conclude that $P(d\Mid u^{(s)};k^{(s)})=0$, as it always contains a zero factor.

(2)
We first show that $P(d\Mid u^{(s)};k^{(s)})$ is of the form \eqref{eq-preliminary-laruent-proper-1} when $d>sc+((m-r)r+\tau)/(n-s)$ and $s\neq n$. Let $U=\{u_1,u_2,\dots,u_s\}$. The part contributing to the degree in $x_{u_s}$ of the numerator of $P(d\Mid u^{(s)};k^{(s)})$ is
\begin{multline}
S_{u,k}\left[ x_0^{\tau-|\mu|}  H_\tau P_\mu\left(\left[\cfrac{q^{c-b}-q^{-d}}{1-q^{c+1}}x_0+X_m+\cfrac{q-q^{c+1}}{1-q^{c+1}}X_{n\setminus m}\right];q,q^{c+1}\right) \right]\\
\times \prod_{\substack{i=1\\ i\notin U}}^n\bigg(\prod_{j=1}^r
\big(q^{k_s-k_j+\chi(u_j>i)}x_{u_s}/x_i\big)_{c+\chi(i\leq m)}
\prod_{j=r+1}^s\big(q^{k_s-k_j+\chi(u_j>i)}x_{u_s}/x_i\big)_c \bigg),
\end{multline}
which is a Laurent polynomial in $x_{u_s}$ of degree at most $(n-s)sc+(m-r)r+\tau$.
The part contributing to the degree in $x_{u_s}$ of the denominator of $P(d\Mid u^{(s)};k^{(s)})$ is
\[
\prod_{\substack{i=1\\ i\notin U}}^n
\big(q^{k_s-d}x_{u_s}/x_i\big)_d,
\]
which has degree $(n-s)d$. Then $P(d\Mid u^{(s)};k^{(s)})$ is of the form \eqref{eq-preliminary-laruent-proper-1} if $d>sc+((m-r)\cdot r+\tau)/(n-s)$. Applying Lemma~\ref{lem-preliminary-laruent-proper} gives
\[
\CT_{x_{u_s}}P(d\Mid u^{(s)};k^{(s)})=
\begin{cases}\displaystyle
\sum_{\substack{u_s<u_{s+1}\leq n\\1\leq k_{s+1}\leq d}}
P(d\Mid u_1,\dots,u_s,u_{s+1};k_1,\dots,k_s,k_{s+1}) \quad &\text{for $u_s<n$,}\\
0 \quad &\text{for $u_s=n$.}
\end{cases}
\]
This completes the proof.
\end{proof}

\begin{cor}\label{cor-mainthm-P-i}
Assume $c\geq b+\mu_1$ and $b\geq \tau+m$. If there exists an $i\in\{n-m+1,\ldots,n-1\}$ such that
\(
ic+b+1 \leq d \leq ic+b+i-n+m+\mu_{n-i},
\)
then
    \begin{equation}\label{eq-mainprop-P(d)-i}
    \CT_{x}P(d)=\sum_{\substack{1\leq u_1<\cdots<u_{i+1}\leq n\\1\leq k_1,\dots,k_{i+1}\leq d}}
    \CT_x P(d\Mid u^{(i+1)};k^{(i+1)}).
    \end{equation}
\end{cor}
\begin{proof}
Under the given conditions, we have 
\(
   nd> nc\geq nb\geq \tau,
\)
and thus the expansion~\eqref{eq-mainthm-P(d)-sum} holds. To complete the proof, we only need to show that the index $s$ in each nonzero term $P(d\Mid u^{(s)};k^{(s)})$ in~\eqref{eq-mainthm-P(d)-sum} can only be $i+1$.

If $s\geq i+2$ then $s\geq n-m+3$. Since $c\geq \mu_1$, we have 
\[
d\leq ic+b+i-n+m+\mu_{n-i}\leq (s-1)c+b+(s-n+m-1).
\]
By the case (1) of Proposition~\ref{prop-mainthm-P(d|u;k)-properties}, we have $P(d\Mid u^{(s)};k^{(s)})=0$.
If $s\leq i$, since $b\geq \tau+m$, we have
\[
    d \geq ic+b+1>sc+\cfrac{(m-r)r+\tau}{n-s}.
\]
It is clear that $s\neq n$ since $s\leq i\leq n-1$. By the case (2) of Proposition~\ref{prop-mainthm-P(d|u;k)-properties}, the summand $\CT\limits_x P(d\Mid u^{(s)};k^{(s)})$ is either zero or can be written as a sum (against that $s$ is maximal). This completes the proof.
\end{proof}

\subsection{Proof of Theorem~\ref{proposition-main-2-2}}\label{subsec-proof2-3}
We complete the proof of Theorem~\ref{proposition-main-2-2} by showing that when $c\geq b+\mu_1$ and $b\geq \tau+m$, 
\[
\CT_x P(d)=0\quad \text{ for $d\in B_2(\delta^m)\cup B_3(\mu)$}.
\]
Here we recall that
\[
B_2(\delta^m)=\bigcup_{i=n-m+1}^{n-1}\{ic+b+1,\ldots,ic+b+\delta^m_{n-i}\},
\]
and
\[
B_3(\mu)=\bigcup_{i=n-m+1}^{n-1}\{ic+b+\delta^m_{n-i}+1,\ldots,ic+b+\delta^m_{n-i}+\mu_{n-i}\}.
\]
\begin{lem}\label{lem-proof2-3}
     Assume $c\geq b+\mu_1$ and $b\geq \tau+m$. The constant term $\CT\limits_x P(d)$ vanishes for $d \in B_2(\delta^m)$.
\end{lem}
\begin{proof}
    Fix $i\in \{n-m+1,\ldots,n-1\}$ and an integer $d$ such that
    \[
    ic+b+1\leq d\leq ic+b+\delta^m_{n-i}=ic+b+i-n+m.
    \]
    By Corollary~\ref{cor-mainthm-P-i}
    \begin{equation}\label{eq-mainthm-P(d)-i-1}
    \CT_{x}P(d)=\sum_{\substack{1\leq u_1<\cdots<u_{i+1}\leq n\\1\leq k_1,\dots,k_{i+1}\leq d}}
    \CT_x P(d\Mid u^{(i+1)};k^{(i+1)}).
    \end{equation}
    Since
    \(
    d\leq ic+b+i-n+m,
    \)
    every summand $\CT\limits_x P(d\Mid u^{(i+1)};k^{(i+1)})$ in~\eqref{eq-mainthm-P(d)-i-1} vanishes by the case (1) of Proposition~\ref{prop-mainthm-P(d|u;k)-properties} with $s=i+1$. Hence, $\CT\limits_x P(d)=0$.
\end{proof}

\begin{lem}\label{lem-proof2-4}
    Assume $c\geq b+\mu_1$ and $b\geq \tau+m$. The constant term $\CT\limits_x P(d)$ vanishes for $d\in B_3(\mu)$. 
\end{lem}
\begin{proof}
Fix $i\in \{n-m+1,\ldots,n-1\}$ and an integer $d$ such that 
\[
ic+b+i-n+m+1\leq d \leq ic+b+i-n+m+\mu_{n-i}.
\]
By Corollary~\ref{cor-mainthm-P-i}
\begin{equation}\label{eq-mainthm-P(d)-i-2}
    \CT_{x}P(d)=\sum_{\substack{1\leq u_1<\cdots<u_{i+1}\leq n\\1\leq k_1,\dots,k_{i+1}\leq d}}
    \CT_x P(d\Mid u^{(i+1)};k^{(i+1)}).
\end{equation}
To show $\CT\limits_x P(d)=0$, it suffices to show that each summand $\CT\limits_x P(d\Mid u^{(i+1)};k^{(i+1)})$ in~\eqref{eq-mainthm-P(d)-i-2} vanishes.

For a given rational function $P(d\Mid u^{(i+1)};k^{(i+1)})$, let $d=ic+b+i-n+m+\mathfrak{e}$ for a positive integer $\mathfrak{e}$ such that $1\leq \mathfrak{e} \leq \mu_{n-i}$. Since 
\[
1\leq k_j\leq d=ic+b+i-n+m+\mathfrak{e}, 
\]
for $j=1,\ldots,i+1$, by Lemma~\ref{lem-preliminary-key} with $s=i+1$, $e=i-n+m+\mathfrak{e}$, and $r$ taken as in~\eqref{definition-mainthm-r}, the $k_j$ can only fall into four cases. If one of the cases (i)-(iii) of Lemma~\ref{lem-preliminary-key} holds, then $P(d\Mid u^{(i+1)};k^{(i+1)})=0$ by the same argument as that in the part (1) of the proof of Proposition~\ref{prop-mainthm-P(d|u;k)-properties}. We then proceed with the proof by assuming that the case (iv) of Lemma~\ref{lem-preliminary-key} holds. In this case, by Lemma~\ref{lem-preliminary-cardi-P_mu} with $s=i+1$, $P(d\Mid u^{(i+1)};k^{(i+1)})$ contains a factor
\begin{align*}
    & S_{u,k} \left[P_\mu^{(q,q^{c+1})}\Big[\cfrac{q^{c-b}-q^{-d}}{1-q^{c+1}}x_0+X_m+\cfrac{q-q^{c+1}}{1-q^{c+1}}X_{n\setminus m}\Big]\right]\\
    & = P_\mu^{(q,q^{c+1})}\Big[ -\cfrac{1-q}{1-q^{c+1}}\big(\alpha_1+\cdots+\alpha_{\mathfrak{e}-1}\big)+\big(\beta_1+\cdots+\beta_{n-i-1}) \Big],
\end{align*}
where each $\alpha_k$ is of the form $x_i q^j$, and each $\beta_k$ is a single-letter alphabet. Since $\mathfrak{e}-1\leq \mu_{n-i}-1$, by Proposition~\ref{prop-Mac-vanish} with $i\mapsto n-i$, we have
\[
P_\mu^{(q,q^{c+1})}\Big[ -\cfrac{1-q}{1-q^{c+1}}\big(\alpha_1+\cdots+\alpha_{\mathfrak{e}-1}\big)+\big(\beta_1+\cdots+\beta_{n-i-1}) \Big]=0.
\]
Therefore, each $P(d\Mid u^{(i+1)};k^{(i+1)})=0$ as it contains a zero factor in all cases. This completes the proof.
\end{proof}

Theorem~\ref{proposition-main-2-2} follows directly from Lemma~\ref{lem-proof2-3} and Lemma~\ref{lem-proof2-4}.

%% file: 8.proof4.tex
\section{Proof of The case (1) of Theorem~\ref{thm-nonmac-factorization}}\label{sec-proof4}

In this section, we give a proof of the case (1) of Theorem~\ref{thm-nonmac-factorization} by Theorem~\ref{proposition-main-2-1}, Theorem~\ref{proposition-main-2-2} and Theorem~\ref{Proposition-main}. 

Note that in this paper we always assume $t=q^c$. Moreover, by~\eqref{eq-nonmac-T_iaction} and~\eqref{def-nonmac-prescibed-partial-sym-antisym}, the polynomial $\mathcal{A}_{t}^{(X_m)}\mathcal{S}_{t}^{(X_{n\setminus m})} E_{(\lambda+\delta^m,0^{n-m})}(x;q,t)$ equals $\mathcal{A}_{t}^{(X_m)} E_{(\lambda+\delta^m,0^{n-m})}(x;q,t)$ up to a constant multiple. Hence, we restate the case (1) of Theorem~\ref{thm-nonmac-factorization} as the following proposition.

\begin{prop}\label{prop-proof4-1}
For a partition $\lambda$ such that $\ell(\lambda)<m$, we have
\[
\mathcal{A}_{q^c}^{(X_m)}E_{(\lambda+\delta^m,0^{n-m})}(x;q,q^c)=c_\lambda \prod_{1\leq i<j\leq m}(q^c x_i-x_j)P_\lambda\left( \left[X_m +\cfrac{q-q^{c+1}}{1-q^{c+1}}X_{n\setminus m}  \right];q,q^{c+1}\right),
\]
where $c_\lambda\in\mathbb{F}$. 
\end{prop}

The proof begins with an observation. Let $\mathrm{Asym}_{q^c}[X_m]$ be the subspace of $\mathbb{F}[x_1,\ldots,x_n]$ that are $q^c$-antisymmetric ($t$-antisymmetric with $t=q^c$) in the first $m$ variables. That is
\[
\mathrm{Asym}_{q^c}[X_m]:=\{f(x)\in\mathbb{F}[x_1,\ldots,x_n]:T_if(x)=-f(x) \text{ for $i=1,\ldots,m-1$}\},
\]
where $T_i$ is the operator defined in~\eqref{definition-nonmac-T_i} with $t=q^c$. This subspace admits a natural orthogonal basis (with respect to the scalar product $\langle-,-\rangle_{n}$) given by\cite{baratta}
\[
\{\mathcal{A}_{q^c}^{(X_m)} E_{(\mu+\delta^m,\rho)}(x;q,q^c):\mu\in\mathcal{P}_{m},\rho\in\mathbb{N}^{n-m}\}.
\]
It is also clear that the polynomial
\begin{equation}\label{def-proof4-vp}
P_{m,\lambda}(x;q,q^c):=\prod_{1\leq i<j\leq m}(q^c x_i-x_j)P_\lambda\left( \left[X_m +\cfrac{q-q^{c+1}}{1-q^{c+1}}X_{n\setminus m}  \right];q,q^{c+1}\right)
\end{equation}
lies in $\mathrm{Asym}_{q^c}[X_m]$. Therefore, to prove Proposition~\ref{prop-proof4-1}, it suffices to show that the above polynomial is orthogonal to every basis element $\mathcal{A}_{q^c}^{(X_m)} E_{(\mu+\delta^m,\rho)}(x;q,q^c):\mu\in\mathcal{P}_{m},\rho\in\mathbb{N}^{n-m}$ except when $\mu=\lambda,\rho=0^{n-m}$.

\begin{lem}\label{lem-proof4-1}
    Let $\nu\in\mathbb{N}^n$, $\lambda$ be a partition such that $\ell(\lambda)<m$ and $P_{m,\lambda}(x;q,q^c)$ be the polynomial defined in~\eqref{def-proof4-vp}. Then
    \[
    \langle P_{m,\lambda}(x;q,q^c),\mathcal{A}_{t}^{(X_m)} E_\nu(x;q,q^c)\rangle_{q,q^c,n}\neq 0
    \]
    only when $\nu^+=\lambda+\delta^m$.
\end{lem}

To prove the above lemma, for nonnegative integers $a,b,c$, a composition $\nu\in\mathbb{N}^n$, and a partition $\lambda$ such that $\ell(\lambda)<m$, define a generalized $q$-Morris constant term $T_{n,m}(a,b,c,\nu,\lambda)$ as
\begin{multline}
\CT_{x,x_0} \prod_{i=1}^{n}(x_0/x_i)_a(qx_i/x_0)_b \prod_{1\leq i<j\leq n}(x_i/x_j)_c(qx_j/x_i)_c \mathcal{A}_{q^{-c}}^{(X_m)} E_\nu(x^{-1};q^{-1},q^{-c})\\
\times \prod_{1\leq i<j\leq m}(q^c x_i-x_j)P_\lambda\left( \left[\cfrac{q^{c-b}-q^a}{1-q^{c+1}}x_0+X_m +\cfrac{q-q^{c+1}}{1-q^{c+1}}X_{n\setminus m}  \right];q,q^{c+1}\right).
\end{multline}
Note that, by homogeneity, $T_{n,m}(a,b,c,\nu,\lambda)\neq 0$ only when $|\lambda|+\binom{m}{2}=|\nu|$, and
\[
    \Big\langle P_{m,\lambda}(x;q,q^c), \mathcal{A}_{q^c}^{(X_m)}E_\nu(x;q,q^c) \Big\rangle_{q,q^c,n}\\=T_{n,m}(0,0,c,\nu,\lambda)=T_{n,m}(a,0,c,\nu,\lambda).
\]
The following proposition gives a necessary condition for $T_{n,m}(0,0,c,\nu,\lambda)\neq 0$.

\begin{prop}\label{prop-prooflem-T_R}
If $T_{n,m}(0,0,c,\nu,\lambda)\neq 0$, then
\begin{equation}\label{eq-prooflem-T-R}
\cfrac{T_{n,m}(a,b,c,\nu,\lambda)}{T_{n,m}(0,b,c,\nu,\lambda)}=\cfrac{R_{n,m}(a,b,c,\nu,\lambda)}{R_{n,m}(0,b,c,\nu,\lambda)}
\end{equation}
holds for all nonnegative integers $a,b,c$, where
\[
R_{n,m}(a,b,c,\nu,\lambda)=\prod_{i=0}^{n-1} \cfrac{(q^{ic+1+a})_{b+\chi(i\geq n-m)(i-n+m+\lambda_{n-i})}}{(q^{ic+1+a})_{\nu^+_{n-i}}}.
\]
\end{prop}

\begin{proof}
The proof begins by verifying that both sides of~\eqref{eq-prooflem-T-R} are
\begin{enumerate}[(i)]
    \item polynomials in $q^a$ of degree at most $nb$ for fixed $b$ and $c$ such that $b\geq \nu_1^+$; 
    \item rational functions in $q^b$ for fixed $a$ and $c$;
    \item rational functions in $q^c$ for fixed $a$ and $b$;
    \item vanishing for
    \[
      -a\in B_1(\nu^+)\cup B_2(\delta^m)\cup B_3(\lambda),
    \]
    provided that $b$ and $c$ satisfy $c\geq b+\lambda_1$, $b\geq |\nu|+m$. Here the sets $B_i$ are defined as in~\eqref{def-sets-B1-B3}.
\end{enumerate}
All the statements are straightforward for the right-hand side. To establish the results for the left-hand side, we first find that $T_{n,m}(a,b,c,\nu,\lambda)$ is a constant multiple of certain $F_{n,m}(a,b,c,H_\tau,\lambda)$ (defined in~\eqref{def-intro-F(mu,H)}). Let \( \tau = |\nu| - \binom{m}{2} \), and 
\[
H_\tau=\cfrac{\mathcal{A}_{q^{-c}}^{(X_m)}E_\nu(x^{-1};q^{-1},q^{-c})}{\prod_{1\leq i<j\leq m}(q^{-c}x_i^{-1}-x_j^{-1})}.
\]
It is clear that $H_\tau$ is of the form~\eqref{def-intro-H_tau} and symmetric in $x_1,\ldots,x_m$. Define
\begin{multline*}
    h(x):=\prod_{i=1}^{n}(x_0/x_i)_a(qx_i/x_0)_b \prod_{1\leq i<j\leq n}(x_i/x_j)_c(qx_j/x_i)_c\prod_{1\leq i<j\leq m}(x_i^{-1}-q^c x_j^{-1})\\
    \times H_\tau P_\lambda\left( \left[\cfrac{q^{c-b}-q^a}{1-q^{c+1}}x_0+X_m +\cfrac{q-q^{c+1}}{1-q^{c+1}}X_{n\setminus m}  \right];q,q^{c+1}\right).
\end{multline*}
Then $h(x)$ is antisymmetric in $x_1,\ldots,x_m$ and
\[
    T_{n,m}(a,b,c,\nu,\lambda)=\CT_x \left(\prod_{1\leq i<j\leq m}(x_i-q^{-c}x_j) h(x)\right).
\]
Using Proposition~\ref{prop-preliminary-antisymmetric} twice, we obtain
\begin{multline}
    T_{n,m}(a,b,c,\nu,\lambda)=\cfrac{[m]_{q^{-c}}!}{[m]_{q^{c+1}}!}\cdot \CT_x \left(\prod_{1\leq i<j\leq m}(x_i-q^{c+1}x_j)h(x)\right)\\
    = \cfrac{[m]_{q^{-c}}!}{[m]_{q^{c+1}}!}\cdot F_{n,m}(a,b,c,H_\tau,\lambda).
\end{multline}
By Theorem~\ref{proposition-main-2-1}, we know that the statements (i)-(iii) hold for the left-hand side of ~\eqref{eq-prooflem-T-R}. By Theorem~\ref{proposition-main-2-2} we also know that when $c\geq b+\lambda_1$, $b\geq |\nu|+m$, the constant term $T_{n,m}(a,b,c,\nu,\lambda)$ vanishes for $-a\in B_2(\delta^m)\cup B_3(\lambda)$. On the other hand, by Remark~\ref{rem-nonmac}, it is not hard to see that $T_{n,m}(a,b,c,\nu,\lambda)$ is a linear combination of those $M_n(a,b,c,\eta,\alpha)$ with $\eta^+=\nu^+,\alpha\in\mathbb{N}^n,|\alpha|\leq |\lambda|+\binom{m}{2}$. It follows from Theorem~\ref{Proposition-main} that $T_{n,m}(a,b,c,\nu,\lambda)$ vanishes for $-a\in B_1(\nu^+)$. This completes all the verification.

Let $\mathcal{S}=B_1(\nu^+)\cup B_2(\delta^m)\cup B_3(\lambda)\cup\{0\}$. Then equality~\eqref{eq-prooflem-T-R} holds for all triples $(a,b,c)$ such that $-a\in\mathcal{S}, c\geq b+\lambda_1, b\geq |\nu|+m$. By Lemma~\ref{lem-preliminary-a,b,c}, the equality~\eqref{eq-prooflem-T-R} holds for all nonnegative integers $a,b,c$.
\end{proof}

\begin{proof}[Proof of Lemma~\ref{lem-proof4-1}]
By Proposition~\ref{prop-prooflem-T_R}, 
\[
  \langle P_{m,\lambda}(x;q,q^c),\mathcal{A}_{t}^{(X_m)} E_\nu(x;q,q^c)\rangle_{q,q^c,n}\neq 0
\]
implies that
\[
\cfrac{T_{n,m}(a,b,c,\nu,\lambda)}{T_{n,m}(0,b,c,\nu,\lambda)}=\cfrac{R_{n,m}(a,b,c,\nu,\lambda)}{R_{n,m}(0,b,c,\nu,\lambda)}
\]
holds for all nonnegative integers $a,b,c$. By taking $b=0$ in the above identity, we obtain
\[
    \cfrac{R_{n,m}(a,0,c,\nu,\lambda)}{R_{n,m}(0,0,c,\nu,\lambda)}=\cfrac{T_{n,m}(a,0,c,\nu,\lambda)}{T_{n,m}(0,0,c,\nu,\lambda)}=1.
\]
It is now not hard to check that the equality
\[
    R_{n,m}(a,0,c,\nu,\lambda)=R_{n,m}(0,0,c,\nu,\lambda)
\]
holds for all nonnegative integers $a$ if and only if $\nu^+=(\lambda+\delta^m,0^{n-m})$.

\end{proof}

We now prove Proposition~\ref{prop-proof4-1}.

\begin{proof}[Proof of Proposition~\ref{prop-proof4-1}]
  For $\mu\in\mathcal{P}_{m}$ and $\nu\in\mathbb{N}^{n-m}$, by Lemma~\ref{lem-proof4-1}
  \[
    \langle P_{m,\lambda}(x;q,q^c),\mathcal{A}_{t}^{(X_m)} E_{(\mu+\delta^m,\nu)}(x;q,q^c)\rangle_{q,q^c,n}\neq 0
   \]
  only when $(\mu+\delta^m,\nu)^+=(\lambda+\delta^m,0^{n-m})$. This forces that exactly $n-m+1$ zeros appear in $(\mu+\delta^m,\nu)$. Therefore, $\nu=0^{n-m}$ and $\mu=\lambda$. 
\end{proof}

%% file: 7.proof3.tex
\section{Proof of Theorem~\ref{main-thm}}\label{sec-proof3}
In this section we complete the proof of Theorem~\ref{main-thm}. As discussed in Section~\ref{section-two families}, we have completed Steps (1) and (2) outlined in the introduction by Theorem~\ref{thm-nonmac-factorization}, Theorem~\ref{proposition-main-2-1}, Theorem~\ref{proposition-main-2-2} and Theorem~\ref{Proposition-main}. To complete the Step (3), in Subsection~\ref{subsec-proof3-1} we give an explicit closed-form expression for the constant term $F_{n,m}(a,b,c,\lambda,\xi,\mu)$ at $-a=(n-m)c+b+1$, under the assumption that $c\geq b+\mu_1, b\geq |\lambda|+|\xi|+m$. Finally, in Subsection~\ref{subsec-proof3-2} we prove Theorem~\ref{main-thm}.

\subsection{The additional point}\label{subsec-proof3-1}
Throughout this subsection, let $b$ and $c$ be nonnegative integers such that $c\geq b+\mu_1, b\geq |\lambda|+|\xi|+m$. Fix $\tau=|\lambda|+|\xi|$ and
\[
H_\tau=P_\lambda\left(\left[X_m^{-1}+\cfrac{1-q^c}{1-q^{c+1}}X_{n\setminus m}^{-1}\right];q,q^{c+1} \right)P_\xi\left( X_{n\setminus m}^{-1};q^{c+1},q^c\right).
\]
Recall the definition of $P(d)$ and $P(d\Mid u^{(s)};k^{(s)})$ in~\eqref{def-P(d)} and~\eqref{definition-mainthm-P(d|u;k)}, respectively. Then
\[
\CT\limits_x P(d)=F_{n,m}(-d,b,c,H_\tau,\mu)=F_{n,m}(-d,b,c,\lambda,\xi,\mu).
\]
Hence, to give an explicit closed-form expression for $F_{n,m}(-(n-m)c-b-1,b,c,\lambda,\xi,\mu)$, it suffices to give an expression for $\CT\limits_x P((n-m)c+b+1)$.

The essential steps are displayed in Lemmas~\ref{lem-additional-lem1}-\ref{lem-additional-lem3}.
\begin{lem}\label{lem-additional-lem1}
Let $d_0=(n-m)c+b+1$. Then
\begin{equation}
    \underset{x}{\CT} P(d_0)=\sum_{u_1=1}^m
    \underset{x}{\CT} P\big(d_0\Mid \widetilde{u};\widetilde{k}\big),
\end{equation}
where $\widetilde{u}=\big(u_1,m+1,\ldots,n)$ and $\widetilde{k}=\big((n-m)c+b+1,\cdots,c+b+1,b+1\big)$.
\end{lem}
\begin{proof}
It is clear that $nd_0>nb>|\lambda|+|\xi|=\tau$, and so we have the expansion
\begin{equation}\label{eq-mainthm-extra-P-sum-1}
\CT_{x}P(d_0)=\sum_{s\in T\subseteq \{1,\dots,n\}}\sum_{\substack{1\leq u_1<\cdots<u_s\leq n\\1\leq k_1,\dots,k_s\leq d_0}}
\underset{x}{\CT} P(d_0\Mid u^{(s)};k^{(s)}),
\end{equation}
where the $s$ in each $\CT\limits_x P(d_0\Mid u^{(s)};k^{(s)})$ is maximal. By the same argument as that in the proof of Corollary~\ref{cor-mainthm-P-i}, the $s$ in \eqref{eq-mainthm-extra-P-sum-1} can only be $n-m+1$. Hence, the equation reduces to
\begin{equation}\label{eq-mainthm-extra-P-sum-2}
\CT_{x} P(d_0)=\sum_{\substack{1\leq u_1<\cdots<u_{n-m+1}\leq n\\1\leq k_1,\dots,k_{n-m+1}\leq d_0}}
\underset{x}{\CT} P(d_0\Mid u^{(n-m+1)};k^{(n-m+1)}).
\end{equation}
Recall that for a given $P(d_0\Mid u^{(n-m+1)};k^{(n-m+1)})$, the corresponding $r=r(u^{(n-m+1)})$ is defined as in~\eqref{definition-mainthm-r}. When $s=n-m+1$, we have $r\geq 1$. 

If $r>1$, by Lemma~\ref{lem-preliminary-key} with $s=n-m+1$ and $e=1<r$, only three cases can hold for the $k_i$:
\begin{enumerate}
\item[(i)] $1\leq k_i\leq b$ for some $i$ with $1\leq i\leq s$;
\item[(ii)] $-c-1\leq k_i-k_j\leq c$ for some $(i,j)$ such that $1\leq i<j\leq r$;
\item[(iii)] $-c\leq k_{i}-k_{j}\leq c-1$ for some $(i,j)$ such that $1\leq i<j\leq s$ and $j>r$.
\end{enumerate}
By the same argument as that in the first part of the proof of Proposition~\ref{prop-mainthm-P(d|u;k)-properties}, each of these cases will make $P\big(d_0\Mid u;k\big)=0$. Hence, the only non-vanishing case for $P(d_0\Mid u^{(n-m+1)};k^{(n-m+1)})$ is the $r=1$ case. This forces $u_1\leq m$ and $m<u_2<\ldots<u_{n-m+1}\leq n$. Therefore, 
\[
u_i=i+m-1\quad \text{ for $i=2,\dots,n-m+1$}.
\]
On the other hand, by Corollary~\ref{cor-preliminary-key}, with $s=n-m+1$ and $e=1$, the rational function $P\big(d_0\Mid u^{(n-m+1)};k^{(n-m+1)}\big)$ will vanish unless 
\[
k_i=(n-m+1-i)c+b+1\quad \text{ for $i=1,\ldots,n-m+1$}.
\]
This completes the proof.
\end{proof}

For $1\leq u_1\leq m$, let $\widetilde{u}$ and $\widetilde{k}$ be defined as in Lemma~\ref{lem-additional-lem1}, and let $\tau_{u_1}$ be the substitution $x_n\mapsto q^{(n-m)c}x_{u_1}$.  It is clear that
\begin{equation}\label{eq-additional-P-tauP}
\underset{x}{\CT} P\big(d_0\Mid \widetilde{u};\widetilde{k}\big) = \underset{x}{\CT} \tau_{u_1}\Big( P(d_0\Mid\widetilde{u};\widetilde{k})\Big).
\end{equation}
By a direct calculation, we obtain an explicit expression for $\underset{x}{\CT} \tau_{u_1}\big(P(d_0\Mid\widetilde{u};\widetilde{k})\big)$.
\begin{lem}\label{lem-additional-lem2}
    Let $1\leq u_1\leq m$, $d_0=(n-m)c+b+1$, and $\widetilde{u}$, $\widetilde{k}$ be defined as in Lemma~\ref{lem-additional-lem1}. Then $\underset{x}{\CT} \tau_{u_1}\big(P(d_0\Mid\widetilde{u};\widetilde{k})\big)$ can be expressed as 
\begin{multline}
C_1\times \underset{x}{\CT} \  (q^{(n-m)c+b}x_{u_1})^{|\lambda|-|\mu|} \prod_{\substack{ i=1\\ i\neq u_1}}^m \cfrac{(q^{-b-(n-m)c}x_i/x_{u_1})_{(n-m)c+b} (q^{1+c}x_{u_1}/x_i)_{(n-m)c}}{(x_{u_1}/x_i)_{(n-m)c+b+1}} 
\\
\times P_\mu\Big( \Big[\sum_{\substack{i=1\\ i\neq u_1}}^m x_i \Big];q,q^{c+1}\Big)P_\lambda\Big(\Big[ \sum_{\substack{i=1\\ i\neq u_1}}^m x_i^{-1} +\cfrac{q^b-q^{(n-m+1)c+b+1}}{1-q^{c+1}}(q^{(n-m)c+b}x_{u_1})^{-1} \Big];q,q^{c+1} \Big) 
\\
\times \prod_{1\leq i < j\leq m} (x_i/x_j)_{c+1}(qx_j/x_i)_{c+1},
\end{multline}
where
\begin{equation}\label{eq-additional-C_1-expression}
    C_1=q^{(b+1)|\xi|+|\lambda|-|\mu|}\cdot \prod_{i=0}^{n-m} \cfrac{(q^{1+c})_{ic} (q^{-ic-b})_b }{(q)_{ic+b}} P_{\xi}\Big(\Big[\cfrac{1-q^{(n-m)c}}{1-q^c}\Big];q^{c+1},q^c\Big).
\end{equation}
\end{lem}

For partitions $\lambda$ and $\mu$, denote
\begin{multline}
    A_n^\prime(a,b,c,\lambda,\mu)= \underset{x}{\CT} x_0^{|\lambda|-|\mu|}\prod_{i=1}^n (x_0/x_i)_a (qx_i/x_0)_b \prod_{1\leq i\neq j \leq n}(x_i/x_j)_c \\
    \times P_\lambda( x^{-1};q, q^c) P_\mu\left( \left[ \cfrac{q^{c-b-1}-q^a}{1-q^c}x_0+\sum_{i=1}^n x_i\right];q,q^c \right).
\end{multline}
The next lemma serves as a bridge. It connects the rational function $P(d_0)$ with the AFLT-type $q$-Morris constant term.
\begin{lem}\label{lem-additional-lem3}
Let $d_0=(n-m)c+b+1$. Then
    \begin{equation}
    \begin{aligned}
        \underset{x}{\CT} P(d_0)&=\cfrac{C_1}{(m-1)!}\prod_{i=1}^{m-1}\cfrac{1-q^{(i+1)(c+1)}}{1-q^{c+1}}A_{m-1}^\prime((n-m+1)c+b+1,c-b,c+1,\mu,\lambda)\\
        &=C_1\times\cfrac{1-q^{m(c+1)}}{1-q^{c+1}} A_{m-1}((n-m+1)c+b+1,c-b,c+1,\mu,\lambda),
        \end{aligned}
    \end{equation}
    where $C_1$ is the constant defined in~\eqref{eq-additional-C_1-expression}.
\end{lem}
\begin{proof}
For $u_1=1,\ldots,m$, let
\begin{multline}
R_{u_1}(x):= (q^{(n-m)c+b}x_{u_1})^{|\lambda|-|\mu|} \prod_{\substack{ i=1\\ i\neq u_1}}^m \cfrac{(q^{-b-(n-m)c}x_i/x_{u_1})_{(n-m)c+b} (q^{1+c}x_{u_1}/x_i)_{(n-m)c}}{(x_{u_1}/x_i)_{(n-m)c+b+1}} 
\\
\times  P_\mu\Big( \Big[\sum_{\substack{i=1\\ i\neq u_1}}^m x_i \Big];q,q^{c+1}\Big)  P_\lambda\Big(\Big[ \sum_{\substack{i=1\\ i\neq u_1}}^m x_i^{-1} +\cfrac{q^b-q^{(n-m+1)c+b+1}}{1-q^{c+1}}(q^{(n-m)c+b}x_{u_1})^{-1} \Big];q,q^{c+1} \Big).
\end{multline}

By Lemma~\ref{lem-additional-lem2}, for $\widetilde{u}=(u_1,m+1,\ldots,n)$, $\widetilde{k}=\big((n-m)c+b+1,\cdots,c+b+1,b+1\big)$,
\[
\underset{x}{\CT} \tau_{u_1}\Big(P(d_0\Mid\widetilde{u};\widetilde{k})\Big)= C_1 \times \underset{x}{\CT} R_{u_1}(x) \prod_{1\leq i<j\leq m}(x_i/x_j)_{c+1}(qx_j/x_i)_{c+1}.
\]
By Lemma~\ref{lem-additional-lem1} and ~\eqref{eq-additional-P-tauP}, we can write
    \[
        \underset{x}{\CT} P(d_0)= C_1 \times \underset{x}{\CT} \sum_{u_1=1}^m R_{u_1}(x) \prod_{1\leq i<j\leq m}(x_i/x_j)_{c+1}(qx_j/x_i)_{c+1}.
    \]
Notice that $\sum_{u_1=1}^m R_{u_1}(x)$ is a Laurent polynomial symmetric in $x_1,\ldots,x_m$. By Proposition~\ref{prop-preliminary-laruent-equiv}, we have
\begin{align*}
\underset{x}{\CT} P(d_0) &= \cfrac{C_1}{m!}\prod_{i=1}^{m-1}\cfrac{1-q^{(i+1)(c+1)}}{1-q^{c+1}}\times \underset{x}{\CT} \sum_{u_1=1}^m R_{u_1}(x)\prod_{1\leq i\neq j\leq m}(x_i/x_j)_{c+1}\\
&=\cfrac{C_1}{m!}\prod_{i=1}^{m-1}\cfrac{1-q^{(i+1)(c+1)}}{1-q^{c+1}}\times \sum_{u_1=1}^m \underset{x}{\CT} R_{u_1}(x)\prod_{1\leq i\neq j\leq m}(x_i/x_j)_{c+1}.
\end{align*}
A straightforward calculation yields
\begin{align*}
&\underset{x}{\CT}  R_{u_1}(x)\prod_{1\leq i\neq j\leq m}(x_i/x_j)_{c+1}\\
& = \underset{x}{\CT}  (q^{(n-m)c+b}x_{u_1})^{|\lambda|-|\mu|} P_\mu\Big( \Big[\sum_{\substack{i=1\\ i\neq u_1}}^m x_i \Big];q,q^{c+1}\Big)\\
& \times P_\lambda\Big(\Big[ \sum_{\substack{i=1\\ i\neq u_1}}^m x_i^{-1} +\cfrac{q^b-q^{(n-m+1)c+b+1}}{1-q^{c+1}}(q^{(n-m)c+b}x_{u_1})^{-1} \Big];q,q^{c+1} \Big) \\
& \times\prod_{\substack{i=1\\ i\neq u_1}}^m (q^{-b-(n-m)c} x_i/x_{u_1})_{(n-m+1)c+b+1} (q^{b+(n-m)c+1}x_{u_1}/x_i)_{c-b} \prod_{\substack{1\leq i\neq j\leq m \\ i,j\neq u_1}}(x_i/x_j)_{c+1}.
\end{align*}
One can verify that
\begin{align*}
\underset{x}{\CT} R_{u_1}(x) \prod_{1\leq i\neq j\leq m}(x_i/x_j)_{c+1} 
&= \underset{x}{\CT} \Big[ R_{u_1}(x)\prod_{1\leq i\neq j\leq m}(x_i/x_j)_{c+1}\Big]_{q^{(n-m)c+b}x_{u_1}\mapsto x_0}\\
&= A_{m-1}^\prime((n-m+1)c+b+1,c-b,c+1,\mu,\lambda).
\end{align*}
Hence,
\begin{align*}
    \underset{x}{\CT} P(d_0)
    &= \cfrac{C_1}{m!}\prod_{i=1}^{m-1}\cfrac{1-q^{(i+1)(c+1)}}{1-q^{c+1}}\times \sum_{u_1=1}^m \underset{x}{\CT} R_{u_1}(x)\prod_{1\leq i\neq j\leq m}(x_i/x_j)_{c+1}\\
    &=\cfrac{C_1}{m!}\prod_{i=1}^{m-1}\cfrac{1-q^{(i+1)(c+1)}}{1-q^{c+1}}\times \sum_{u_1=1}^m A_{m-1}^\prime((n-m+1)c+b+1,c-b,c+1,\mu,\lambda)\\
    &= \cfrac{C_1}{(m-1)!}\prod_{i=1}^{m-1}\cfrac{1-q^{(i+1)(c+1)}}{1-q^{c+1}}\times A_{m-1}^\prime((n-m+1)c+b+1,c-b,c+1,\mu,\lambda).
\end{align*}
The second equality of the lemma follows from Proposition~\ref{prop-preliminary-laruent-equiv}.
\end{proof}

By Lemma~\ref{lem-additional-lem3} and~\eqref{eq-intro-AFLT}, together with the fact that
\[
\underset{x}{\CT} P(d)=F_{n,m}(-d,b,c,\lambda,\xi,\mu),
\]
we obtain an explicit expression for $F_{n,m}(a,b,c,\lambda,\xi,\mu)$ at $a=-(n-m)c-b-1$. We conclude the result in the next lemma.

\begin{lem}
Assume $c\geq b+\mu_1$, $b\geq |\lambda|+|\xi|+m$ and $a=-(n-m)c-b-1$. Then
\begin{multline}
    F_{n,m}(a,b,c,\lambda,\xi,\mu) = (-1)^{|\mu|}q^{(b+1)(|\xi|+|\lambda|)+\sum_{i=1}^{\ell(\mu)}\binom{\mu_i}{2}-(c+1)n(\mu)} \cfrac{1-q^{m(c+1)}}{1-q^{c+1}} \prod_{i=0}^{n-m} \cfrac{(q^{1+c})_{ic} (q^{-ic-b})_b }{(q)_{ic+b}}\\
    \times \prod_{i=1}^{m-1}\prod_{j=1}^{m-1}(q^{(i-j)(c+1)-b-\mu_i+\lambda_{j+1}})_{\lambda_j-\lambda_{j+1}}
    \prod_{i=1}^{m-1} \cfrac{(q^{(n-i)(c+1)+b-(n-m-1)+\mu_i})_{c-b-\mu_i} (q)_{i(c+1)}}{(q)_{i(c+1)-b-1-\mu_i+\lambda_1} (q)_{c+1}}\\
    \times P_\xi\Big(\Big[\cfrac{1-q^{(n-m)c}}{1-q^c}\Big];q^{c+1},q^c\Big)P_\lambda\Big(\Big[ \cfrac{1-q^{nc+m}}{1-q^{c+1}}\Big];q,q^{c+1} \Big) P_\mu\Big( \Big[ \cfrac{1-q^{(m-1)(c+1)}}{1-q^{c+1}}\Big];q,q^{c+1}\Big) .
\end{multline}
\end{lem}

\subsection{Proof of Theorem~\ref{main-thm}}\label{subsec-proof3-2}
In this subsection, we complete the proof of Theorem~\ref{main-thm}. Throughout the proof, we assume $\lambda$ and $\xi$ satisfy at least one of the conditions in Theorem~\ref{main-thm}. 

\begin{proof}
In the previous sections, we have established the polynomiality, rationality, and vanishing properties of $F_{n,m}(a,b,c,\lambda,\xi,\mu)$. We have also given an expression for $F_{n,m}(a,b,c,\lambda,\xi,\mu)$ at $a=-(n-m)c-b-1$ when $c\geq b+\mu_1$, $b\geq |\lambda|+|\xi|+m$. Let $R_{n,m}(a,b,c,\lambda,\xi,\mu)$ be the right-hand side of~\eqref{eq-intro-maineq}. It is trivial to check that:
\begin{enumerate}
        \item for fixed nonnegative integers $b$ and $c$ such that $b\geq \lambda_1+m$, $R_{n,m}(a,b,c,\lambda,\xi,\mu)$ is a polynomial in $q^a$ of degree at most $nb+|\mu|-|\lambda|-|\xi|$;
        \item for fixed nonnegative integers $a$ and $c$, $R_{n,m}(a,b,c,\lambda,\xi,\mu)$ is a rational function in $q^b$;
        \item for fixed nonnegative integers $a$ and $b$, 
        \[
        (q)_c^n/(q)_{nc} \cdot R_{n,m}(a,b,c,\lambda,\xi,\mu)
        \]
        is a rational function in $q^c$;
        \item $R_{n,m}(a,b,c,\lambda,\xi,\mu)$ equals $F_{n,m}(a,b,c,\lambda,\xi,\mu)$ when $c\geq b+\mu_1$, $b\geq |\lambda|+|\xi|+m$, and
        \[
        -a\in B_1(\eta^+)\cup B_2(\delta^m)\cup B_3(\mu)\cup \{(n-m)c+b+1\}.
        \]
        Here $\eta=(\lambda+\delta^m,\xi)$, and the sets $B_i:i=1,2,3$ are defined in~\eqref{def-sets-B1-B3}.
\end{enumerate}
Using Lemma~\ref{lem-preliminary-a,b,c}, we can extend
\[
F_{n,m}(a,b,c,\lambda,\xi,\mu)=R_{n,m}(a,b,c,\lambda,\xi,\mu)
\]
to all nonnegative integers $a,b,c$.
\end{proof}

%% file: 9.appendix.tex
\section{Proof of Lemma~\ref{cor-nonmac-splitvariables}}
The proof follows an argument similar to that for symmetric Macdonald polynomials in~\cite[Section 6.7]{Mac95}. We first recall a Cauchy-type identity for nonsymmetric Macdonald polynomials, due to Mimachi and Noumi. For variables $x = (x_1, \ldots, x_n)$ and $y = (y_1, \ldots, y_n)$, denote
\begin{equation*}
\Omega(x;y|q,t)=\prod_{1\leq j < i \leq n}\frac{(qtx_iy_j)_{\infty}}{(qx_iy_j)_{\infty}}\prod_{1\leq i \leq n}\frac{(qtx_iy_i)_{\infty}}{(x_iy_i)_{\infty}}\prod_{1\leq i < j \leq n}\frac{(tx_iy_j)_{\infty}}{(x_iy_j)_{\infty}}.
\end{equation*}

\begin{thm}\cite[Theorem A]{cauchy-formula}
The function $\Omega(x;y|q,t)$ can be expanded in terms of nonsymmetric Macdonald polynomials:
\begin{equation}\label{eq-nonmac-cauchy-identity}
\Omega(x;y|q,t)=\sum_{\lambda\in\mathbb{N}^n}a_{\lambda}(q,t)E_{\lambda}(x;q,t)E_{\lambda}(y;q^{-1},t^{-1}),
\end{equation}
where for each composition $\lambda\in\mathbb{N}^n$, the coefficient $a_{\lambda}(q,t)$ is explicitly given in~\cite[(0.3)]{cauchy-formula}.
\end{thm}

Let $I:=\{i_1,\ldots,i_r\}$ and $J:=\{j_1,\ldots,j_k\}$ be a set partition of $\{1,\ldots,n\}$, i.e., $I\cup J=\{1,\ldots,n\}$ and $I$, $J$ are disjoint. For a formal power series $f(y)=f(y_1,\ldots,y_n)$, denote
\begin{equation}
    f(y_I)= f(y_1,\ldots,y_n)\Mid_{y_{j_1}=\cdots=y_{j_k}=0}, \quad f(y_J)= f(y_1,\ldots,y_n)\Mid_{y_{i_1}=\cdots=y_{i_r}=0}.
\end{equation}
It is trivial to check that
\begin{equation}\label{eq-nonmac-product-kernel}
    \Omega(x;y_I | q,t)\Omega(x;y_J | q,t)=\Omega(x;y | q,t).
\end{equation}
For a fixed positive integer $n$ define
\begin{equation}
    f_{\mu\nu}^\lambda:=\cfrac{\langle E_\mu(x;q,t)E_\nu(x;q,t), E_\lambda(x;q,t)\rangle_{n}}{\langle E_\lambda(x;q,t),E_\lambda(x;q,t)\rangle_{n}}.
\end{equation}
By the orthogonality of nonsymmetric Macdonald polynomials, it is equivalent to write
\begin{equation}\label{eq-nonmac-f_munu-version2}
    E_\mu(x;q,t)E_\nu(x;q,t)=\sum_\lambda f_{\mu\nu}^\lambda E_\lambda(x;q,t).
\end{equation}

\begin{lem}\label{lem-nonmac-splitvariables}
    For $\lambda\in\mathbb{N}^n$, the nonsymmetric Macdonald polynomial $E_\lambda(y;q^{-1},t^{-1})$ has the following expansion:
    \begin{equation*}
        E_\lambda(y;q^{-1},t^{-1})=\sum_{\mu,\nu\in\mathbb{N}^n} \cfrac{a_\nu(q,t) a_\mu(q,t)}{a_\lambda(q,t)} f_{\mu\nu}^\lambda E_\mu(y_I;q^{-1},t^{-1}) E_\nu(y_J;q^{-1},t^{-1}).
    \end{equation*}
\end{lem}
\begin{proof}
    By~\eqref{eq-nonmac-cauchy-identity}, ~\eqref{eq-nonmac-product-kernel}, and~\eqref{eq-nonmac-f_munu-version2}, we have
    \begin{align*}
        & \sum_{\lambda\in\mathbb{N}^n} a_\lambda(q,t)E_\lambda(x;q,t) E_\lambda(y;q^{-1},t^{-1})\\
      & = \Omega(x;y| q,t)\\
      & = \Omega(x;y_I | q,t) \Omega(x;y_J | q,t)\\
      & = \Big(\sum_{\mu\in\mathbb{N}^n} a_\mu(q,t)E_\mu(x;q,t) E_\mu(y_I;q^{-1},t^{-1}) \Big)\Big(\sum_{\nu\in\mathbb{N}^n} a_\nu(q,t)E_\nu(x;q,t) E_\nu(y_J;q^{-1},t^{-1}) \Big)\\
      & = \sum_{\mu,\nu\in\mathbb{N}^n} a_\mu(q,t)a_\nu(q,t) E_\mu(y_I;q^{-1},t^{-1})  E_\nu(y_J;q^{-1},t^{-1}) E_\mu(x;q,t)E_\nu(x;q,t)\\
      & = \sum_{\mu,\nu,\lambda\in\mathbb{N}^n} a_\mu(q,t)a_\nu(q,t) f_{\mu\nu}^\lambda E_\mu(y_I;q^{-1},t^{-1}) E_\nu(y_J;q^{-1},t^{-1}) E_\lambda(x;q,t).
     \end{align*}
    The result therefore follows by comparing the coefficients of $E_\lambda(x;q,t)$ on both sides.
\end{proof}

To obtain a vanishing condition for $f_{\mu\nu}^\lambda$, we define another partial order $\ll$ on compositions by $\mu \ll \lambda$ if there exists a permutation $\pi$ such that for all $1\leq i\leq n$,
\begin{align*}
    &\mu_i<\lambda_{\pi(i)} \quad \text{for $i<\pi(i)$},
 \\   &\mu_i\leq \lambda_{\pi(i)} \quad \text{for $i\geq \pi(i)$}.
\end{align*}
It is immediate that $\mu \ll \lambda$ implies $\mu^+ \subseteq \lambda^+$. For a positive integer $p$ and $1\leq i_1<\cdots<i_p\leq n$, the nonsymmetric Macdonald polynomials admit the following Pieri-type formula\cite[(A.7)]{skew-super}.
\begin{equation}\label{eq-nonmac-pieri-rule}
    x_{i_1}\cdots x_{i_p}E_\mu(x;q,t)=\sum_{\lambda \in \mathbb{J}_{n,I,\mu}}c_{I \mu}^{\lambda}(q^{-1},t^{-1})E_\lambda(x;q,t),
\end{equation}
where $I=\{i_1,\ldots,i_p\}$, $c_{I \mu}^{\lambda}(q^{-1},t^{-1})\in \mathbb{F}$ and 
\[
\mathbb{J}_{n,I,\mu}=\{\lambda:\mu \ll \lambda \ll \mu+\mathbf{1}^n, |\lambda|=|\mu|+|I|\}.
\]
For $\alpha\in\mathbb{N}^n$ with $r=\max\{\alpha_i:1\leq i\leq n\}$, by repeatedly applying~\eqref{eq-nonmac-pieri-rule}, we have that
\begin{equation}\label{eq-nonmac-pieri-rule-2}
    x_1^{\alpha_1}\cdots x_n^{\alpha_n}E_\mu(x;q,t)=\sum_{\lambda \in \mathbb{J}_{n,\alpha,\mu}}c_{\alpha\mu}^{\lambda}(q^{-1},t^{-1})E_\lambda(x;q,t),
\end{equation}
where $c_{\alpha\mu}^{\lambda}(q^{-1},t^{-1})\in\mathbb{F}$ and
\[
\mathbb{J}_{n,\alpha,\mu}=\{\lambda:\mu\ll \lambda \ll \mu+r\cdot\mathbf{1}^n, |\lambda|=|\mu|+|\alpha|\}.
\]
Using the Pieri-type formula, we have the following vanishing condition for $f_{\mu\nu}^\lambda$.

\begin{lem}\label{lem-nonmac-f_munulambda=0}
    $f_{\mu\nu}^\lambda=0$ if $\mu\not\ll\lambda$ or $\nu\not\ll\lambda$.
\end{lem}
\begin{proof}
Assume $\mu\not\ll\lambda$. Consider the expansion of $E_\mu(x;q,t)E_\nu(x;q,t)$ in the basis of nonsymmetric Macdonald polynomials. By~\eqref{eq-nonmac-uppertriangular}, we have
\[
E_\mu(x;q,t)E_\nu(x;q,t)=\sum_{\eta\preceq\nu} b_{\nu\eta}x^\eta E_\mu(x;q,t).
\]
For each $\eta\preceq\nu$, formula~\eqref{eq-nonmac-pieri-rule-2} implies that the expansion of $x^\eta E_\mu(x;q,t)$ in the nonsymmetric Macdonald basis only involves terms $E_\gamma(x;q,t)$ with $\mu\ll\gamma$. Since $\mu\not\ll\lambda$, the polynomial $E_\lambda(x;q,t)$ does not occur in this expansion. Therefore, by the orthogonality of nonsymmetric Macdonald polynomials,
\[
\langle E_\mu E_\nu, E_\lambda \rangle_n=\sum_{\eta\preceq \nu} b_{\nu\eta} \langle x^\eta E_\mu,E_\lambda\rangle_n=0.
\]
Hence, we have $f_{\mu\nu}^\lambda=0$. The case $\nu\not\ll\lambda$ follows in the same way by the symmetry $f_{\mu\nu}^\lambda=f_{\nu\mu}^\lambda$.
\end{proof}

Let $x^{(i)}:=(x_1,\ldots,x_{i-1},x_{i+1},\ldots,x_n)$. 
\begin{lem}\label{lem-nonmac-x_i=0}
    Let $\nu=(\nu_1,\ldots,\nu_n)$ be a composition. Then for $i=1,\ldots,n$, 
    \begin{equation}\label{eq-nonmac-x_i=0}
        E_\nu(x;q,t)\Mid_{x_i=0}=
        \begin{cases}
            \text{a linear combination of $E_\eta(x^{(i)};q,t):\eta\in\mathbb{N}^{n-1},\ \eta^+=\nu^+$} & \text{if $\ell(\nu^+)<n$},\\
            0 &\text{if $\ell(\nu^+)=n$}.
        \end{cases}
    \end{equation}
\end{lem}
\begin{proof}
    The second case follows directly from (\ref{eq-nonmac-uppertriangular}). We now prove the first case by descending induction on $i$. The base case is $i=n$, which follows by (\ref{eq-nonmac-stability}). Now we assume ~\eqref{eq-nonmac-x_i=0} holds for $i=k$. When $i=k-1$, by (\ref{eq-nonmac-T_iaction}) we have
    \begin{equation}\label{eq-nonmac-x_i=0-1}
        T_{k-1}E_\nu(x_1,\ldots,x_n)=c_\nu E_\nu(x_1,\ldots,x_n)+ c_\nu^\prime E_{s_{k-1}\nu}(x_1,\ldots,x_n),
    \end{equation}
    where $c_\nu$ and $c_\nu^\prime$ are constants. By definition of $T_{k-1}$, we also have 
    \begin{equation}\label{eq-nonmac-x_i=0-2}
        T_{k-1}E_\nu(x_1,\ldots,x_n)=t E_\nu(x_1,\ldots,x_n)
        +\cfrac{tx_{k-1}-x_k}{x_{k-1}-x_k}\Big(s_{k-1} E_\nu(x_1,\ldots,x_n)-E_\nu(x_1,\ldots,x_n)\Big).
    \end{equation}
    By taking $x_k=0$ in both (\ref{eq-nonmac-x_i=0-1}) and (\ref{eq-nonmac-x_i=0-2}), and by the inductive assumption, we have
    \begin{multline}\label{eq-nonmac-x_i=0-3}
       t E_\nu(x_1,\ldots,x_{k-2},0,x_{k-1},x_{k+1},\ldots,x_n)=
       c_\nu E_\nu(x_1,\ldots,x_n)\Mid_{x_k=0}+c_\nu^\prime E_{s_{k-1}\nu}(x_1,\ldots,x_n)\Mid_{x_k=0}\\
        = \text{a linear combination of $E_\eta(x^{(k)};q,t):\eta^+=\nu^+$}.
    \end{multline}
    By changing variables $x_{k-1} \leftrightarrow x_k$ in (\ref{eq-nonmac-x_i=0-3}), we obtain
    \[
    E_\nu(x_1,\ldots,x_{k-2},0,x_k,x_{k+1},\ldots,x_n)= \text{a linear combination of $E_\eta(x^{(k-1)};q,t):\eta^+=\nu^+$}.
    \]
    Hence~\eqref{eq-nonmac-x_i=0} holds when $i=k$ is replaced by $i=k-1$.
\end{proof}

Combining Lemma \ref{lem-nonmac-splitvariables}, Lemma \ref{lem-nonmac-f_munulambda=0} and Lemma \ref{lem-nonmac-x_i=0}, we obtain Lemma~\ref{cor-nonmac-splitvariables}.